\documentclass[11pt]{amsart}
\usepackage{fullpage}
\usepackage{color}
\usepackage{pstricks,pst-node,pst-plot} %This loads the pstricks package
\usepackage{graphicx,psfrag} %,epstopdf}
\usepackage{color} %This loads the color package
\usepackage{tikz}
\usepackage{pgffor}
\usepackage{hyperref}

\usepackage{perpage}	 %the perpage package
	
\MakePerPage{footnote}	 %the perpage package command

\newtheorem{theorem}{Theorem}[section]
\newtheorem{lemma}[theorem]{Lemma}
\newtheorem{corollary}[theorem]{Corollary}
\newtheorem{proposition}[theorem]{Proposition}

\newtheorem{sublemma}[theorem]{Sublemma}
\newtheorem{conjecture}[theorem]{Conjecture}

\theoremstyle{definition}
\newtheorem{definition}[theorem]{Definition}
\newtheorem{example}[theorem]{Example}

\newtheorem{question}[theorem]{Question}

\theoremstyle{remark}
\newtheorem{remark}[theorem]{Remark}

\numberwithin{equation}{section}

\definecolor{gray}{rgb}{.5,.5,.5}

\definecolor{black}{rgb}{0,0,0}

\definecolor{blue}{rgb}{0,0,1}

\definecolor{red}{rgb}{1,0,0}

\definecolor{green}{rgb}{0,1,0}

\definecolor{yellow}{rgb}{1,1,.4}

\newrgbcolor{purple}{.5 0 .5}
\newrgbcolor{black}{0 0 0}
\newrgbcolor{white}{1 1 1}
\newrgbcolor{gold}{.5 .5 .2}
\newrgbcolor{darkgreen}{0 .5 0}
\newrgbcolor{gray}{.5 .5 .5}
\newrgbcolor{lightgray}{.75 .75 .75}

\begin{document}

\title{Kauffman's clock lattice as a graph of perfect matchings:  a formula for its height}
%\title{Kauffman's clock lattice as a graph of perfect matchings:  its height and structure}
%\title{Computing the height of Kauffman's clock lattice}
%Upper bounds on clock number and bridge number of a knot}

\author{Moshe Cohen}
\address{Department of Mathematics and Computer Science, Bar-Ilan University, Ramat Gan 52900, Israel}
\email{cohenm10@macs.biu.ac.il}

\author{Mina Teicher}
\address{Department of Mathematics and Computer Science, Bar-Ilan University, Ramat Gan 52900, Israel}
\email{teicher@macs.biu.ac.il}

\date{August 2012}
%\thanks{%This research was supported by the VIGRE grant at LSU, DMS-0739382}

\begin{abstract}
We give an algorithmic computation for the height of Kauffman's clock lattice obtained from a knot diagram with two adjacent regions starred and without crossing information specified.
%{\red Abe \cite{Abe} defines the clock number $p(K)$ of a knot $K$ to be the minimum over all diagrams of the height of the clock lattice obtained from a knot diagram.}  
We show that this lattice is more familiarly the graph %$\mathcal{G}$
 of perfect matchings of a bipartite graph %$\Gamma$
  obtained from the knot diagram by overlaying the two dual Tait graphs %$G$ 
 of the knot diagram.  % and its dual.  % $G^*$.  
%We obtain upper bounds for the clock number $p(K)$ %and bridge number $br(K)$
% of the knot from the combinatorics of $\Gamma$.
%
This setting also makes evident applications to Chebyshev or harmonic knots, whose related bipartite graph is the popular grid graph, and to discrete Morse functions.  Furthermore we prove structural properties of the bipartite graph in general.
%Abstract for Comb Sem:  A knot is a circle embedded in three-space, but we immediately translate it into a bipartite graph Gamma following previous work of the author.  We consider the graph G of perfect matchings of this graph Gamma, providing a combinatorial formula for the width of G.  We present several properties of Gamma and construct a partition of its vertices into cycles that relate directly to G.
\end{abstract}

\maketitle

%\keywords{...}

%MSC2010: Knots and links in S^3
%MSC2010: Relations with graph theory

%MSC2010: Factorization, matching, partitioning, covering and packing
%MSC2010: Applications

%NOTE:  Maybe write $\Gamma=\Gamma(G,v,f)$ and $\mathcal{G}=\mathcal{G}(\Gamma)$.
%Lemma about tree-like property of cycles and leaves?

\section{Introduction}
\label{sec:Intro}

%Kauffman's clock move lattice of a knot diagram (used to obtain the Alexander polynomial) is in fact the graph of perfect matchings of the dimer graph coming from this diagram (used to obtain several knot polynomials by the author).  Recently Abe defined the clock number of a knot to be one more than the minimum over all diagrams of the height of this lattice, giving the crossing number as a lower bound and showing equality for two-bridge knots.  In the setting below, a succinct formula is given for this number on any knot diagram using graph theoretic techniques, which reveal a structure related to the bridge number of the diagram.

There is a bijection between the set of all knot (and link) diagrams and the set of all signed plane graphs $G$.  Spanning tree expansions of $G$ have been used to produce several models of use in knot theory:  Kauffman \cite{Kauff} gives one for the Alexander polynomial; Thistlethwaite \cite{Th} for the Jones polynomial (related to work on the signed Tutte polynomial by Kauffman \cite{Kauff:sign} and extended to work on the Bollob{\'a}s-Riordan-Tutte polynomial by Dasbach, Futer, Kalfagianni, Lin, and Stoltzfus \cite{DaFuKaLiSt}); Greene \cite{Greene:tr} (of a different flavor) for the Heegaard Floer homology of the branched double cover of a knot; Ozsv{\'a}th and Szab{\'o} \cite{OzSz:tr} and Baldwin and Levine \cite{BalLev} (of this different flavor; analogous to different unpublished work by Ozsv{\'a}th and Szab{\'o}) give some for knot Floer homology;  and Wehrli \cite{We:sp}, Champarnekar and Kofman \cite{ChKo:sp} (independently), and Roberts \cite{Rob:tr} (of this different flavor; see also remarks by Jaeger \cite{Jag:tr}) for Khovanov homology.
%  Spanning tree models for knot homology theories come in two flavors:  Ozsv{\'a}th and Szab{\'o} {\red ???} give one for knot Floer homology while Champarnekar and Kofman \cite{ChKo:sp} and independently Wehrli \cite{We:sp} give some for Khovanov homology; and Roberts {\red ???} gives one for Khovanov homology (also see remarks by Jaeger {\red ???}), Greene {\red ???} gives one for the Heegaard Floer homology of a branched double cover of a knot, and more recently Baldwin and Levine \cite{BalLev} give one for knot Floer homology (analogous to unpublished work by Ozsv{\'a}th and Szab{\'o}).

There is a another bijection between the set of (rooted) spanning trees (or arborescences) of a plane graph $G$ and the set of perfect matchings (or dimer coverings) of a related plane bipartite graph $\Gamma$ that has been explored in previous work by the first author \cite{Co:jones} and the first author with Dasbach and Russell \cite{CoDaRu}, as well as in work by Kenyon, Propp, and Wilson \cite{KenPropp}, who say about this bijection:
%who extended the result of Temperley (1974) for square grid graphs.  This original result also appeared as problem 4.30 in \cite[pp. 34,104,243-244]{Lov}.  Independent generalizations for unweighted, undirected graphs were produced by Burton and Pemantle \cite{BurPem} for infinite graphs and by F. Y. Wu,
\begin{quotation}
This theorem, along with its proof, is a generalization of a result of Temperley (1974) which is discussed in problem 4.30 of \cite[pp. 34, 104, 243-244]{Lov}.  %(Lovasz, 1979, pages 34, 104, 243-244).  
The unweighted undirected generalization was independently discovered by Burton and Pemantle \cite{BurPem}, %(1993), 
who applied it to infinite graphs, and also by F. Y. Wu, who included it in lecture notes for a course.
\end{quotation}
The related graph $\widehat{\Gamma}$ also appeared in work by Huggett, Moffatt, and Virdee \cite{HugVir}.  The graph $\Gamma$ is currently being studied by Kravchenko and Polyak \cite{KraPol} for knots on a torus in relation to cluster algebras.  %This graph $\Gamma$ also appeared in work by Huggett, Moffatt, and Virdee \cite{HugVir}.  It is currently being studied by Kravchenko and Polyak \cite{KraPol} for knots on a torus in relation to cluster algebras.  %resistor networks and the Poisson bracket.  %{\red In current work Polyak [???] is using these ideas for applications to resistor networks, Poisson brackets, ...  on a torus...}
Dimers themselves have been studied extensively, as well;  see for example Kenyon's lecture notes \cite{Ken} on the subject.  

By the end of this present paper the authors hope that the reader will prefer the perfect matching model for $\Gamma$ below to the spanning tree model for $G$.  In support of this we offer evidence that previous work in knot theory can be translated into concepts that are more regularly studied by graph theorists.

The primary example of this considered below is Kauffman's clock lattice $\mathcal{L}$ \cite{Kauff}, which we translate into the graph $\mathcal{G}$ of perfect matchings of the plane bipartite graph $\Gamma$.  This perspective offers beneficial insight to both sides:  well-studied combinatorial tools can now be applied to knots, and some basic topological structure makes $\mathcal{G}$ easier to understand by directing its edges.  %the edges of $\mathcal{G}$ directed, making it easier to understand.
  In short this analogy allows for the height of the lattice to be seen as the diameter of $\mathcal{G}$, a topic of interest in work by Hernando, Hurtado, and Noy \cite{HerHurNoy} and Athanasiadis and Roichman \cite{AthYuv}.

The work in this present paper is done for knot projections without crossing information; this corresponds to the unweighted graph $\Gamma$.  One may obtain this crossing information by weighting the graph in one of several different ways, including Kauffman's ``black and white holes'' and the first author's previous work on this subject.  It is currently unclear to to the authors whether there is one weighting that is more useful than all the others in every context.

\bigskip

\textbf{Results.}  The main result of this paper, Theorem \ref{conj:sum}, states that this height can be computed combinatorially from the graph $\Gamma$ by counting the number of (square) faces within certain cycles $\{C_i\}$ that emerge in the discussion below.  Specifically these cycles are constructed in Theorem \ref{prop:concentriccircles} and arise from the unique minimum and unique maximum elements in the clock lattice $\mathcal{L}$ in Theorem \ref{thm:clocked}.

Subection \ref{subsec:harmonic} gives an application of the main result to grid graphs, which appear often in graph theory literature.  These are actually the balanced overlaid Tait graphs for \emph{harmonic knots}.% studied by Koseleff and Pecker \cite{KosPec1}, \cite{KosPec4}, \cite{KosPec3}.  See further references in Subsection \ref{subsec:harmonic}.

%A \emph{harmonic curve (or Chebyshev curve)} is one that admits a parametrization whose three coordinate functions $x=x_1,y=x_2,z=x_3$ are the classical Chebyshev polynomials $T_{x_i}(t)$ defined by $T_n(\cos t)=\cos (nt)$.  Identifying the ends of a non-singular harmonic curve, one obtains a \emph{harmonic knot} $H(x_1,x_2,x_3)$ if and only if the three parameters $x_i$ are pairwise coprime integers (Comstock 1897 \cite{Comstock} or see also \cite{KosPec3} or \cite{FF}).

%Koseleff and Pecker found in \cite{KosPec1} that the trefoil could be parametrized in such a way, leading them to study harmonic knots in \cite{KosPec4}, \cite{KosPec3}, and \cite{KosPec}.  Harmonic knots are polynomial analogues of the famous Lissajous knots studied in \cite{BDHZ}, \cite{BHJS}, \cite{Crom}, \cite{HosZir}, \cite{JonPrz}, \cite{Lam}, \cite{Lam:dis}; however, the figure-eight knot is not a Lissajous knot but is the harmonic knot $H(3,5,7)$.

%Furthermore, 
Conjecturally, the authors of the present paper believe there is a relationship between the number of cycles of $\{C_i\}$ and the minimum number of local maxima in this particular knot diagram.  Of course this provides an upper bound on the bridge number of the knot.  This relationship is motivated by Abe's two main results in \cite{Abe}, discussed at the end of Section \ref{sec:PerfectMatchingGraph}, and reaffirmed by new work on harmonic knots.

%\begin{itemize}
%	\item TO INCLUDE:  Abe and 2-bridge -- Need MORE on Abe with better introduction
%	\item grid graphs and harmonic knots
%	\item graph thoerists' work \cite{HerHurNoy}, \cite{AthYuv}
%	\item $\Gamma$ also studied in \cite{HugVir}
%	\item note on unweighted graph -- crossings by weightings
%	\item OUTLINE OF THE PAPER, including Theorem \ref{thm:blackvertices}
%\end{itemize}

Lastly, Subsection \ref{subsec:DMT} and specifically Proposition \ref{prop:DMT} establishes the correspondence between perfect matchings of $\Gamma$ and discrete Morse functions on a 2-complex of the 2-sphere whose 1-skeleton is the (unsigned) plane graph $G$ with a pair of specifically chosen critical cells.%  In particular, {\red result says} there is a correspondence between perfect matchings of $\Gamma$ and discrete Morse functions on a 2-complex with a pair of specifically chosen critical cells.  This relationship might be used to produce knotting in higher dimensions.

%The paper is organized as follows:

\bigskip

\textbf{Organization.}  The next Section \ref{sec:BOTG} gives the construction for the balanced overlaid Tait graph $\Gamma$ and introduces the Periphery Proposition \ref{prop:PropertyK} which appears several times throughout the paper.  The graph $\mathcal{G}$ of perfect matchings of $\Gamma$ is defined in Section \ref{sec:PerfectMatchingGraph}.  Some useful notions from graph theory are discussed in Section \ref{sec:graphtheory}:  connectivity and elementary graphs.

The main results begin in Section \ref{sec:partition} with some operations that are used in the main construction Theorem \ref{prop:concentriccircles}.  Further structural properties like \emph{leaves}, \emph{accordions}, and \emph{party hats} are discussed in Section \ref{sec:2valent} (specifically Theorem \ref{thm:blackvertices}), and some reduction moves are introduced to simplify $\Gamma$.  The proof of the Main Theorem \ref{conj:sum} in Section \ref{sec:graphofpm} is split into several lemmas based on these moves.%reduction moves.

Finally several examples are discussed in Section \ref{sec:examples}, including a subsection on harmonic knots and a subsection on discrete Morse theory.

\bigskip

\textbf{Acknowledgements.} The first author was partially supported by the Oswald Veblen Fund and by the Minerva Foundation of Germany.  Inspiration for this project arose from three places:  the Combinatorics group at Bar-Ilan University and specifically a seminar talk by Roy Ben-Ari on part of his Masters of Science thesis \cite{BenAri} under the supervision of Ron Adin and Yuval Roichman; a preprint \cite{Abe} by Yukiko Abe of Tokyo Institute of Technology containing some results of her Masters thesis; and the first author's graduate work \cite{CoDaRu, Co:jones} at Louisiana State University on the balanced overlaid Tait graph together with his familiarity with Kauffman's clock lattice.  The first author would also like to thank LSU VIGRE for sponsoring the first Baton Rouge Young Topologists Research Retreat in January 2012 whose central theme was the graph $\Gamma$ discussed below, Kate Kearney with whom he co-organized the workshop, and Cody Armond who contributed to many helpful conversations throughout.
%The first author would also like to acknowledge LSU VIGRE for sponsoring the first Baton Rouge Young Topologists Research Retreat, with whom he co-organized 
%The first author would also like to acknowledge participants of the first Baton Rouge Young Topologists Research Retreat, sponsored by LSU VIGRE, whose central theme was the first graph discussed below, and co-organized with Kate Kearney, during which conversations with Cody Armond were particularly helpful.
%The first author would like to expressly thank Kate Kearney, with whom he co-organized the first Baton Rouge Young Topologists Research Retreat (sponsored by the LSU Topology Group with support from VIGRE) whose central theme was the graph discussed below, as well as Cody Armond for many productive conversations throughout this workshop.

\section{The balanced overlaid Tait graph}
\label{sec:BOTG}
%\section{Background}
%\label{sec:Background}

A \emph{knot} $K$ is a circle $S^1$ embedded in $S^3=\mathbb{R}^3\cup\{\infty\}$.  A \emph{link} is the embedding of several copies of $S^1$.  A knot or link \emph{diagram} $D$ is the projection of the knot or link onto $\mathbb{R}^2$ with under- and over-crossing information.  A theorem by Reidemeister in 1926 (see for example \cite{Lic}) states that two diagrams represent the same knot if and only if there is a sequence of the three Reidemeister moves taking one diagram to the other.

The knot diagram considered without crossing information is a 4-regular plane graph called the \emph{projection graph} (or the \emph{universe} $U$ according to Kauffman \cite{Kauff}).  By Euler's formula, there are two more faces than vertices; Kauffman chooses two adjacent faces to omit and marks these $*$ with stars.  He then considers \emph{states}:  bijections between the set of vertices and the set of all un-starred faces.  The state itself is depicted by placing markers at a corner of each crossing and in each face.

%The knot diagram considered without crossing information is a 4-regular plane graph called the \emph{projection graph}.  Kauffman \cite{Kauff} calls this a \emph{universe} $U$ and considers \emph{states} on this universe:  bijections between the set of crossings and the set of all but two of the faces.  The two omitted faces must be adjacent; they are starred $*$ in the figures.  The state itself is depicted by placing markers at a corner of each crossing and in each face.

Ultimately, Kauffman uses the states on a universe to produce the Alexander polynomial $\Delta_K(t)$ of the knot $K$, which since its finding in 1923 has remained one of the most important classical knot invariants.  It is precisely due to the Alexander module that the two starred faces must be adjacent.  There are at least sixteen equivalent definitions of the Alexander polynomial, some of this redundancy owing perhaps to the different ways in which we can define Kauffman's states.  One such way involving the construction below can be found in \cite{CoDaRu}.

A state can be realized as a rooted spanning tree of a plane graph $G$ obtained from a diagram together with the complementary rooted spanning tree of the plane dual $G^*$ to this graph.  One can obtain this signed \emph{Tait graph} $G$ from a diagram by checkerboard-coloring its regions, taking the black regions to be the vertices, and taking signed edges corresponding to the crossings as in Figure  \ref{fig:SignedTaitGraphSignsTIKZ}.

\begin{figure}[h]
\begin{center}

\begin{tikzpicture}
	\fill[gray!50!white] (0,0) -- +(.5,.5) -- +(1,0) -- cycle;
	\fill[gray!50!white] (0,1) -- +(.5,-.5) -- +(1,0) -- cycle;
	\draw[-] (1,0) -- +(-1,1);
	\fill[color=white] (.4,.4) rectangle +(.2,.2);
	\draw[-] (0,0) -- +(1,1);

	\fill[gray!50!white] (2,0) -- +(.5,.5) -- +(1,0) -- cycle;
	\fill[gray!50!white] (2,1) -- +(.5,-.5) -- +(1,0) -- cycle;
	\draw[-] (2,0) -- +(1,1);
	\fill[color=white] (2.4,.4) rectangle +(.2,.2);
	\draw[-] (3,0) -- +(-1,1);

	\draw (-1.5,.5) node {positive};
	\draw (4.5,.5) node {negative};
	
\end{tikzpicture}
	\caption{Crossings determine the sign of the edges in the signed Tait graph.}
	\label{fig:SignedTaitGraphSignsTIKZ}
\end{center}
\end{figure}
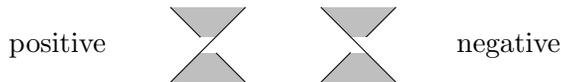

Overlaying $G$ with its plane dual $G^*$ (which is the Tait graph corresponding to the white regions), one obtains the \emph{overlaid Tait graph} $\widehat{\Gamma}$.  This graph is bipartite:  its black vertex set corresponds to the intersections of an edge of $G$ with its dual edge in $G^*$, and its white vertex set corresponds to the vertices of both the Tait graph $G$ and its plane dual $G^*$.  That is, $V(\widehat{\Gamma})=[E(G)\cap E(G^*)]\sqcup [V(G)\sqcup V(G^*)]$.  The edges of this graph are the half-edges of both Tait graphs.  A similar notion is found in work by Huggett, Moffatt, and Virdee \cite{HugVir}.

All of the black vertices of $\widehat{\Gamma}$ are four-valent, as these correspond with vertices of the universe $U$, and all of the faces of $\widehat{\Gamma}$ are square, as these correspond to edges of the universe $U$, as in Fig. \ref{fig:SquareFace}.

\begin{figure}[h]
\begin{center}
\begin{pspicture}(0,0)(4,2)

\pcline[linewidth=1 pt, linecolor=lightgray]{->}(0,1)(4,1)
\pcline[linewidth=1 pt, linecolor=lightgray]{-}(1,0)(1,2)
\pcline[linewidth=1 pt, linecolor=lightgray]{-}(3,0)(3,2)

\pscircle[linewidth=1pt, linecolor=black, fillstyle=solid, fillcolor=black](1,1){.15}
\pscircle[linewidth=1pt, linecolor=black, fillstyle=solid, fillcolor=black](3,1){.15}

\pcline[linewidth=1 pt]{-}(1,1)(2,0)
\pcline[linewidth=1 pt]{-}(3,1)(2,0)
\pcline[linewidth=1 pt]{-}(1,1)(2,2)
\pcline[linewidth=1 pt]{-}(3,1)(2,2)

\pscircle[linewidth=1pt, linecolor=black, fillstyle=solid](2,0){.15}
\pscircle[linewidth=1pt, linecolor=black, fillstyle=solid](2,2){.15}

\end{pspicture}
	\caption{A square face of the overlaid Tait graph $\widehat{\Gamma}$.}% associated with an edge of the projection graph of a knot diagram.}
	\label{fig:SquareFace}
\end{center}
\end{figure}
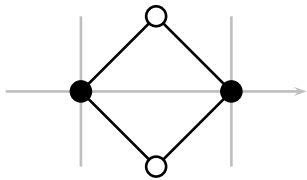

In order to consider perfect matchings, delete the two starred white vertices to obtain the \emph{balanced overlaid Tait graph} $\Gamma$ that is the central graph for the present paper.  For more details of this construction see \cite{Co:jones} %, some of which is reproduced below in Section \ref{sec:construction},
 or consider the following alternative definition.

\begin{definition}
The \emph{balanced overlaid Tait graph} $\Gamma$ is a bipartite graph that can be obtained from a universe $U$ as follows.  Let every four-valent vertex in the universe $U$ be a black vertex in $\Gamma$.  Select two adjacent faces of $U$ and mark them $*$ by stars.  Let every non-starred face of $U$ be a white vertex in $\Gamma$.  A black vertex is adjcent to a white vertex whenever the vertex and face of $U$ are incident.
\end{definition}

\begin{remark}
Since $U$ is a plane graph, so is $\Gamma$.  Furthermore, all faces of $\Gamma$ are square except for the infinite face.  Let the boundary cycle of this infinite face be called the \emph{periphery}.  All black vertices not on the periphery are still four-valent.  The remaining black vertices (exactly those affected by the deletion of the two starred white vertices in $\widehat{\Gamma}$) have valency determined by the following Periphery Proposition.
\end{remark}

%This previous work gives a bijection between rooted spanning trees of $G$ and perfect matchings of a plane bipartite graph $\Gamma$ that comes from overlaying $G$ and $G^*$ and then deleting two vertices.  Let the boundary cycle of the infinite face of $\Gamma$ be called the \emph{periphery}.  Further properties of $\Gamma$ include all faces being square except the infinite face, all black vertices not on the periphery being four-valent, and Property K holding for all vertices on the periphery.
%%, two black vertices on the outer face being two-valent, and the rest of the black vertices on the outer face being three-valent.

A crossing is called \emph{nugatory} if there exists a circle in the projection plane meeting the diagram transversely at that crossing but not meeting the diagram at any other point. Specifically it can be easily removed by twisting some part of the diagram.

\begin{proposition}%[Property K]
\label{prop:PropertyK}
%{\red \cite{Kauff} ???}  
\textbf{Periphery Proposition.}  
The balanced overlaid Tait graph $\Gamma$ for a diagram with no nugatory crossings has the following property:  two of the black vertices on the periphery have valence two; the rest have valence three.
%Of the black vertices on the outer cycle $C$ of the bipartite graph $\Gamma$, two have valence two in $\Gamma$ and the rest have valence three.
\end{proposition}

To make this more obvious we employ two lemmas to show that there can be no black vertices on the periphery of degree one or four.

\begin{lemma}
\label{lem:noblackleaves}
The balanced overlaid Tait graph $\Gamma$ for a diagram with no nugatory crossings has no black leaves.
\end{lemma}

\begin{proof}
Suppose by way of contradiction that there is a black leaf in $\Gamma$.  This black vertex must be four-valent in the overlaid Tait graph $\widehat{\Gamma}$, and so it has three additional edges.  Only two white vertices were deleted from $\widehat{\Gamma}$, and so at least two of these three edges must be incident with the same white vertex.  However the black vertex corresponds to a crossing in the diagram, and so only opposite edges can be incident with the same white vertex.  This results in a nugatory crossing, a contradiction.
\end{proof}

\begin{lemma}
\label{lem:nofourvalentblack}
The balanced overlaid Tait graph $\Gamma$ for a diagram with no nugatory crossings has no four-valent black vertex on the periphery.
\end{lemma}

\begin{proof}
Suppose by way of contradiction that there is a four-valent black vertex $v_1$ on the periphery of $\Gamma$; then there are two white neighbors $u_1$ and $u_2$ of $v_1$ that are also neighbors of $v_1$ on the periphery.  Since all faces of the overlaid Tait graph $\widehat{\Gamma}$ are square, there must be some black vertex $v_2$ such that these vertices form a square face in $\widehat{\Gamma}$ that lives in the infinite face of $\Gamma$.  However, no black vertices were deleted from $\widehat{\Gamma}$ to obtain $\Gamma$, and so $v_2$ must be in $\Gamma$.

This can only be the case when the periphery of $\Gamma$ is itself the square with these four vertices.  Here every face of $\Gamma$ is a square, so by counting the edges around each face $2|E|=4|F|$.  Let $n$ be the number of black vertices; then by counting the edges around each black vertex, $|E|=4n$.  Together these facts give $|F|=2n$.

Since this graph is a plane graph, Euler's formula gives that the number of white vertices must be $n+2$, contradicting the condition that $\Gamma$ is balanced.  In fact, these properties describe $\widehat{\Gamma}$.
\end{proof}

\begin{proof}[Proof of the Periphery Proposition \ref{prop:PropertyK}]
Let $n_i$ be the number of black vertices with valency $i$; by construction $i\leq 4$ and by Lemma \ref{lem:noblackleaves} $i\geq2$.  Since the graph $\Gamma$ is balanced, $|V|=2(n_2+n_3+n_4)$.  Summing the edges around each black vertex we obtain $|E|=\sum in_i=2n_2+3n_3+4n_4$.  Since this is a plane graph, Euler's formula gives $|F|=2-|V|+|E|=2-2n_2-2n_3-2n_4+2n_2+3n_3+4n_4=2+n_3+2n_4$.

By Lemma \ref{lem:nofourvalentblack} there are no four-valent black vertices on the periphery; then its length is $2(n_2+n_3)$.  Summing the edges around each face we obtain $2|E|=\sum if_i=4(|F|-1)+(2(n_2+n_3))(1)=4|F|-4+2n_2+2n_3$.  Substituting for $|F|$ as above, we obtain $n_2=2$.
\end{proof}

\begin{remark}
Although the Periphery Proposition \ref{prop:PropertyK} appears itself to be a slightly unnatural restriction, it follows by the argument above that it comes directly from the more natural conditions of $\Gamma$ being plane bipartite and having black vertices of degree at most four.
\end{remark}

The balanced overlaid Tait graph $\Gamma$ completely determines the universe $U$.  Crossing information can be obtained by choosing a certain weighting on the graph.

\begin{proposition}
\label{prop:UniqueUniverse}
One can obtain a unique universe $U$ from the balanced overlaid Tait graph $\Gamma$.
\end{proposition}

\begin{proof}
Viewing the universe $U$ as a four-valent graph, the edges of this graph correspond to square faces of $\widehat{\Gamma}$, specifically traversing from black vertex to black vertex through the face.

Thus it is enough to show that one can uniquely produce the overlaid Tait graph $\widehat{\Gamma}$ from the balanced overlaid Tait graph $\Gamma$.

Identify the periphery with the unit circle such that the two black vertices on it that have valence two are at position $1$ and $-1$.  Add two new white vertices at $2i$ and $-2i$ that are adjacent to both of these black vertices.  Furthermore, the white vertex at $2i$ (or $-2i$) is adjacent to every black vertex on the upper hemisphere (or lower hemisphere, respectively) of the periphery. 

One can see this is unique because exactly two white vertices need to be added in a planar way adjacent to the two black two-valent vertices on the periphery, and the rest of the black vertices on the periphery are three-valent.
\end{proof}

Thus this is the graph $\Gamma$ that we will consider, given a specific diagram $D$ with two specified adjacent starred regions for a specific knot $K$.

\section{The graph of perfect matchings}%$\mathcal{G}$ of perfect matchings of $\Gamma$}
\label{sec:PerfectMatchingGraph}

We now construct the graph $\mathcal{G}$ of perfect matchings of the bipartite graph $\Gamma$ above.  Unless otherwise specified, we assume $\Gamma$ has the properties as mentioned above and can be obtained from a diagram $D$ of a knot $K$.

We formally take the vertices of $\mathcal{G}$ to be the perfect matchings of $\Gamma$, although the reader may choose to interpret these vertices instead as states of a universe $U$ as in \cite{Kauff} and \cite{Abe}.  An edge in $\mathcal{G}$ corresponds to a \emph{flip move} of perfect matchings, that is, where all but two of the edges of each perfect matching agree, and these four edges create a square face.

The reader may instead consider $\mathcal{G}$ as the \emph{clock lattice} $\mathcal{L}$ constructed in \cite{Kauff} and \cite{Abe}.  Here the edges are directed according to the \emph{clock move} as in Figure \ref{fig:clockmove}.  Given the square face from the flip move as a cycle in the plane oriented counterclockwise, the perfect matching whose edges on this oriented cycle go from white to black is the tail of the directed edge, and the one that goes from black to white is the head.

\begin{figure}[h]
\begin{center}
\begin{tikzpicture}

\draw[lightgray, shift={(0,3)}] (0,1) -- (3,1);
\draw[lightgray, shift={(0,3)}] (.5,0) -- (.5,2);
\draw[lightgray, shift={(0,3)}] (2.5,0) -- (2.5,2);

\draw[shift={(0,3)}] (4,1) node {$\rightarrow$};

\draw[lightgray, shift={(0,3)}] (5,1) -- (8,1);
\draw[lightgray, shift={(0,3)}] (5.5,0) -- (5.5,2);
\draw[lightgray, shift={(0,3)}] (7.5,0) -- (7.5,2);

%MAKE MARKERS NOT JUST DOTS.
%	\fill[color=black, shift={(0,3)}] (.75,1.25) circle (3pt);
%	\fill[color=black, shift={(0,3)}] (2.25,.75) circle (3pt);
%	\fill[color=black, shift={(0,3)}] (5.75,.75) circle (3pt);
%	\fill[color=black, shift={(0,3)}] (7.25,1.25) circle (3pt);
%  Maybe triangles
%\fill[color=black, shift={(0,3)}] (.6,1.1) -- (1,1.1) -- (.6,1.5) -- cycle;
%\fill[color=black, shift={(0,3)}] (2.4,.9) -- (2,.9) -- (2.4,.5) -- cycle;
%\fill[color=black, shift={(0,3)}] (5.6,.9) -- (6,.9) -- (5.6,.5) -- cycle;
%\fill[color=black, shift={(0,3)}] (7.4,1.1) -- (7.4,1.5) -- (7,1.1) -- cycle;
%  Maybe quarter circles?  
\fill[color=black, shift={(0,3)}] (.5,1) -- (1,1) arc (0:90:.5) -- cycle;
\fill[color=black, shift={(0,3)}] (2.5,1) -- (2,1) arc (180:270:.5) -- cycle;
\fill[color=black, shift={(0,3)}] (5.5,1) -- (5.5,.5) arc (270:360:.5) -- cycle;
\fill[color=black, shift={(0,3)}] (7.5,1) -- (7,1) arc (180:90:.5) -- cycle;

\draw[lightgray] (0,1) -- (3,1);
\draw[lightgray] (.5,0) -- (.5,2);
\draw[lightgray] (2.5,0) -- (2.5,2);
\draw (.5,1) -- (1.5,2);
\draw[dashed] (1.5,2) -- (2.5,1);
\draw (2.5,1) -- (1.5,0);
\draw[dashed] (1.5,0) -- (.5,1);

\draw (4,1) node {$\rightarrow$};

\draw[lightgray] (5,1) -- (8,1);
\draw[lightgray] (5.5,0) -- (5.5,2);
\draw[lightgray] (7.5,0) -- (7.5,2);
\draw[dashed] (5.5,1) -- (6.5,2);
\draw (6.5,2) -- (7.5,1);
\draw[dashed] (7.5,1) -- (6.5,0);
\draw (6.5,0) -- (5.5,1);

	\fill[color=white] (1.5,2) circle (3pt);
	\draw (1.5,2) circle (3pt);

	\fill[color=black] (.5,1) circle (3pt);

	\fill[color=white] (1.5,0) circle (3pt);
	\draw (1.5,0) circle (3pt);

	\fill[color=black] (2.5,1) circle (3pt);

	\fill[color=white] (6.5,2) circle (3pt);
	\draw (6.5,2) circle (3pt);

	\fill[color=black] (5.5,1) circle (3pt);

	\fill[color=white] (6.5,0) circle (3pt);
	\draw (6.5,0) circle (3pt);

	\fill[color=black] (7.5,1) circle (3pt);

\end{tikzpicture}
	\caption{The clock move.}
	\label{fig:clockmove}
\end{center}
\end{figure}
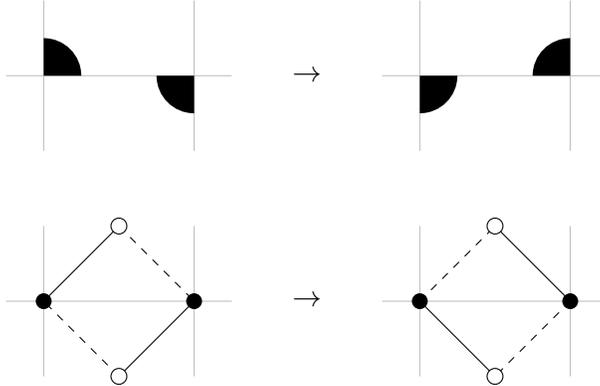

Kauffman proves the following for the clock lattice of a universe.

\begin{theorem}
\label{thm:clocklattice:Kauff}
\cite[Clock Theorem 2.5.]{Kauff}
Let $U$ be a universe and $\delta$ the set of states of $U$ for a given choice of adjacent fixed stars.  Then $\delta$ has a unique clocked state and a unique counterclocked state.  Any state in $\delta$ can be reached from the clocked (counterclocked) state by a series of clockwise (counterclockwise) moves.  Hence any two states in $\delta$ are connected by a series of state transpositions.
\end{theorem}

%$L(U)$ is the lattice of the universe!

Denote the unique minimum by $\widehat{0}$ and the unique maximum by $\widehat{1}$ of the connected lattice $\mathcal{L}$; these are also called the \emph{clocked} and \emph{counterclocked} states, respectively, in the graph $\mathcal{G}$.  Let $h$ be the height of the lattice.

The \emph{diameter} of a graph is the maximum of the shortest distance, or number of edges, between any two vertices taken over all pairs of vertices.

\begin{proposition}
\label{prop:height}
The height $h$ of the clock lattice $\mathcal{L}$ is indeed the diameter of the graph $\mathcal{G}$.
\end{proposition}

\begin{proof}
Since $h$ is the distance between $\widehat{0}$ and $\widehat{1}$, it is enough to show that the distance between any other two elements is no greater than $h$.

Choose any two elements $x$ and $y$ with heights $h(x)$ and $h(y)$, respectively.  Then there are always at least two paths between $x$ and $y$:  one through $\widehat{0}$ and another through $\widehat{1}$.  These two paths have distances $h(x)+h(y)$ and $2h-(h(x)+h(y))$.  Thus if $h(x)+h(y)>h$, the second path is less than $h$.
\end{proof}

We may call $h+1$ the \emph{clock number of the diagram} $p(D)$ for a diagram $D$ with chosen starred regions.  Note that this number is dependent on the actual diagram of the knot given and is not invariant over all diagrams.

%
%
%{\red WITH CHOSEN STARRED REGIONS!  Can I prove invariant of this???}
%
%

To turn this into a knot invariant, Abe \cite{Abe} takes the minimum of $p(D)$ over all diagrams $D$ of a knot $K$ and calls this the \emph{clock number} $p(K)$ of the knot.  The two main theorems of this work by Abe are that $p(K)\geq c(K)$, the crossing number of the knot, with equality when $K$ is a two-bridge knot.  The two-bridge knots are well-understood as the closures of rational tangles.

A \emph{bridge} is one of the arcs in a diagram; thus it consists only of over-crossings.  %The \emph{bridge number} $br(D)$ of the diagram $D$ is the minimum number of disjoint bridges which together include all over-crossings.  
The \emph{bridge index} $br(K)$ of a knot $K$ is the minimum number of disjoint bridges which together include all over-crossings, considering all diagrams.  An equivalent definition for the bridge index uses a Morse function and counts the number of local maxima of the knot, after taking the minimum over all diagrams.% and all Morse functions.

%-------------------------The background above is pretty much good except for adding more details later ---------------

%-------------------------Below is where the work begins---------------------------------------------------------------

\section{Notions from graph theory}
\label{sec:graphtheory}
%\section{Partitioning the vertex set into leaves and cycles}% $C_1,\ldots,C_k$}

%TRANSITION?
%\bigskip

\subsection{Connectivity}
\label{subsec:connect}

The unordered pair $\{A,B\}$ of vertex subsets is a $k$-\emph{separation} of a graph if $A\cup B$ gives the entire vertex set, $|A\cap B|=k$, and the graph has no edge between $A\backslash B$ and $B\backslash A$.  Equivalently a subset $X$ of vertices and edges is said to \emph{separate} two vertex sets $A$ and $B$ if every $A$-$B$ path in the graph contains a vertex or edge from $X$.  The following theorem will be useful below:

\begin{theorem}
\label{thm:Menger}
\cite[Theorem 3.3.1. (Menger 1927)]{Diest}
Let $G=(V,E)$ be a graph and let $A,B\subseteq V$.  Then the minimum number of vertices separating $A$ from $B$ in $G$ is equal to the maximum number of disjoint $A$-$B$ paths in $G$.
\end{theorem}

%MOVED ABOVE!
%A crossing is called \emph{nugatory} if there exists a circle in the projection plane meeting the diagram transversely at that crossing but not meeting the diagram at any other point. Specifically it can be easily removed by twisting some part of the diagram.

A knot $K$ is called \emph{prime} in standard terminology if when it is written as a connect sum $K=K_1\# K_2$, either $K_1$ or $K_2$ must be the unknot.  Below we say that a knot diagram $D$ is \emph{prime-like} if the diagram cannot be written as a conncect sum of diagrams $D=D_1\#D_2$ where there are crossings in both $D_1$ and $D_2$.%  Thus a prime-like diagram must represent a prime knot.

%The following lemma implies the two theorems below.
%{\red CHANGE!  We provide two obvious lemmas omitting their proofs and referring instead to Figure \ref{fig:cutvertices}.  These imply the two theorems below.}

\begin{lemma}
\label{lem:cutvertices}
%\label{lem:1connG}
The following are equivalent for a diagram $D$ for a knot $K$ with no nugatory crossings:
\begin{enumerate}
	\item the diagram $D$ is not prime-like;
	\item the Tait graph $G$ has a cutvertex;
	\item the dual Tait graph $G^*$ has a cutvertex; and
	\item the overlaid Tait graph $\widehat{\Gamma}$ has a 2-separation:  namely the two cutvertices of $G$ and $G^*$ above separate the graph.
\end{enumerate}
Furthermore, there is an arc of the knot diagram incident with the regions associated to both cutvertices.
%A knot diagram $D$ for $K$ with no nugatory crossings is {\red NOT?} prime-like 
%%can be written as the connect sum $D_1 \# D_2$ of two (possibly nontrivially unknotted) knot diagrams 
% if and only if the Tait graph $G$ has a cutvertex.
\end{lemma}

\begin{proof}
The implication $(1)\Rightarrow(2)$ holds by Figure \ref{fig:cutvertices}.  The converse also holds because a circle around one component of the graph meeting only at the cutvertex is the same circle that encloses one of the diagrams in the connect sum.

The implication $(2)\Leftrightarrow(3)$ holds by Figure \ref{fig:cutvertices} and the fact that the Tait graph is unchanged by ambient isotopy of the knot diagram on a sphere.

The implication $(2)\Leftrightarrow(4)$ holds because the circle from above is also the circle that encloses one of the sets in the 2-separation of $\widehat{\Gamma}$ meeting it only at the two cutvertices of $G$ and $G^*$.

Lastly, if the two cutvertices were not incident with a single arc of the knot diagram, then there would be another region in between them, violating all of the above.
\end{proof}

%\begin{lemma}
%\label{lem:cutvertices}
%The Tait graph $G$ has a cutvertex if and only if the dual Tait graph $G^*$ has a cutvertex.  Furthermore, there is an arc of the knot diagram incident with the regions associated to both cutvertices.
%\end{lemma}

%\begin{proof}
%Supposing $G$ has a cutvertex, by Lemma \ref{lem:1connG} the knot diagram $D$ can be written as a connect sum as in the left side of Figure \ref{fig:cutvertices} where the starred vertex of $G$ is a cutvertex.
%
%Then the diagram on the right side represents the same knot after an arc has been pushed through the rest of the diagram.  It is clear that the starred vertex is a cutvertex, but this must now be a vertex in $G^*$.  These two starred vertices represent regions incident with the unchanged arc.
%\end{proof}

%{\red CHECK!  Had $K$ instead of $D$ in the Figure!  Is the proof correct?}

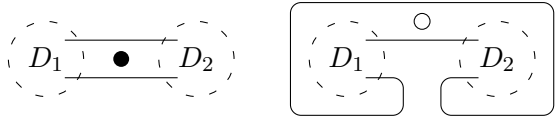
\begin{figure}[h]
\begin{center}
\begin{tikzpicture}

\draw[loosely dashed] (0,0) circle (.5cm);
\draw[loosely dashed] (2,0) circle (.5cm);
\draw[loosely dashed] (4,0) circle (.5cm);
\draw[loosely dashed] (6,0) circle (.5cm);

\draw (0,0) node {$D_1$};
\draw (2,0) node {$D_2$};
\draw (4,0) node {$D_1$};
\draw (6,0) node {$D_2$};

\draw (.25,.25) -- (1.75,.25);
\draw (.25,-.25) -- (1.75,-.25);
\draw (4.25,.25) -- (5.75,.25);

\draw[rounded corners] (4.25,-.25) -- (4.75,-.25) -- (4.75,-.75) -- (3.25,-.75) -- (3.25,.75) -- (6.75,.75) -- (6.75,-.75) -- (5.25,-.75) -- (5.25,-.25) -- (5.75,-.25);

	\fill[color=white] (5,.5) circle (3pt);
	\draw (5,.5) circle (3pt);
	\fill[color=black] (1,0) circle (3pt);

%\draw (1,0) node {$*$};
%\draw (5,.5) node {$*$};

\end{tikzpicture}
	\caption{Cutvertices in $G$ and $G^*$.}%There is a cutvertex in $G$ if and only if there is one in $G^*$.}
	\label{fig:cutvertices}
\end{center}
\end{figure}

%THIS BECOMES THE THEOREM BELOW?
\begin{remark}
\label{rem:obtainingGammaConnectivity}
Observe that since the two cutvertices of $G$ and $G^*$ separate $\widehat{\Gamma}$, the deletion of both of these vertices makes $\Gamma$ disconnected while the deletion of one of these vertices makes $\Gamma$ 1-connected.
\end{remark}

\begin{proposition}
\label{prop:2connG}
The Tait graph $G$ for a prime-like knot diagram with no nugatory crossings is 2-connected.  This also holds for its dual $G^*$.
\end{proposition}

\begin{proof}
The Tait graph $G$ must be connected because the diagram is for a knot with a single component.  By Lemma \ref{lem:cutvertices} there can be no cutvertices.
\end{proof}

The result that $\Gamma$ is 2-connected is Proposition \ref{prop:Gamma2conn} at the end of the next subsection.  In order to show this, we introduce the following possibly unfamiliar idea.

%------------------------------------------HERE I CUT OUT TRYING TO GET GAMMA TO BE 2-CONNECTED!--------------------------

\bigskip

\subsection{Elementary graphs}
\label{subsec:elementary}  
We make specific mention of the following definitions, as they may be unfamiliar to many readers.

\begin{definition}
\label{def:elementary}
An edge of any graph is \emph{allowed} if it lies in some perfect matching of the graph and \emph{forbidden} otherwise.  A graph is \emph{elementary} if its allowed edges form a connected subgraph of the graph.
\end{definition}

Recall that a \emph{vertex covering} is a subset of vertices such that every edge has at least one endpoint in the vertex subset.  Denote by $\nu(X)$ the set of \emph{neighbors} of a subset $X$ of vertices.  Let $K_2$ be the complete graph on two vertices: that is, a single edge.

\begin{theorem}
\label{thm:elementary:LP}
\cite[Theorem 4.1.1.]{LovPlum}
Given a bipartite graph with a bipartition $(U,W)$ of the vertex set, the following are equivalent:
\begin{enumerate}
	\item the graph is elementary;
	\item the graph has exactly two vertex coverings, namely $U$ and $W$;
	\item $|U|=|W|$ and for every non-empty proper subset $X$ of $U$, $|\nu(X)|\geq|X|+1$;
	\item the graph is $K_2$, or there are at least four vertices and for any $u\in U$, $w\in W$, the graph with these two vertices deleted has a perfect matching; and
	\item the graph is connected and every edge is allowed.
\end{enumerate}
\end{theorem}

%USED ABOVE
%Recall that a crossing is called \emph{nugatory} if there exists a circle in the projection plane meeting the diagram transversely at that crossing but not meeting the diagram at any other point. Specifically it can be easily removed by a twisting some part of the diagram.

In order to prove the next main theorem for this subsection, one must remove any nugatory crossings from the knot diagram before taking the associated balanced overlaid Tait graph.

\begin{theorem}
\label{thm:elementary}
The balanced overlaid Tait graph $\Gamma$ for a prime-like knot diagram with no nugatory crossings is an elementary graph.
\end{theorem}

\begin{proof}
Consider some edge $\varepsilon\in E(\Gamma)$ in the balanced overlaid Tait graph.  By Theorem \ref{thm:elementary:LP} (5) we must show that this edge is allowed, that is, that it belongs to some perfect matching of $\Gamma$.  Note that $\varepsilon$ is also an edge in the overlaid Tait graph $\widehat{\Gamma}$ before two white vertices are deleted.

According to %Proposition \ref{prop:DimerTrees} (as in 
\cite[Proposition 4.8]{Co:jones}, there is a bijection between perfect matchings of the balanced overlaid Tait graph $\Gamma$ and rooted spanning trees of one of the Tait graphs $G$ or its dual $G^*$.  Furthermore, this work gives a bijection between the edge $\varepsilon\in E(\Gamma)$ and a directed edge $\vec{e}\in E(G)\cup E(G^*)$ in one of the two Tait graphs, as in Figure \ref{fig:CrossingTaitOverlaidTaitAGAIN}.  

\begin{figure}[h]
\begin{center}
\begin{pspicture}(0,0)(11,4)

\pscircle[linewidth=1pt, linecolor=black, fillstyle=solid, fillcolor=black](1,3){.15}
\pcline[linewidth=.75 pt, linestyle=dotted]{-}(0,3)(1,3)
\pcline[linewidth=.75 pt, linestyle=dotted]{-}(1,3)(2,3)
\pcline[linewidth=.75 pt, linestyle=dotted]{-}(1,2)(1,3)
\pcline[linewidth=.75 pt, linestyle=solid]{-}(1,3)(1,4)
\pscircle[linewidth=1pt, linecolor=black, fillstyle=solid](0,3){.15}
\pscircle[linewidth=1pt, linecolor=black, fillstyle=solid](2,3){.15}
\pscircle[linewidth=1pt, linecolor=black, fillstyle=solid](1,2){.15}
\pscircle[linewidth=1pt, linecolor=black, fillstyle=solid](1,4){.15}

\pscircle[linewidth=1pt, linecolor=black, fillstyle=solid, fillcolor=black](4,3){.15}
\pcline[linewidth=.75 pt, linestyle=dotted]{-}(3,3)(4,3)
\pcline[linewidth=.75 pt, linestyle=dotted]{-}(4,3)(5,3)
\pcline[linewidth=.75 pt, linestyle=solid]{-}(4,2)(4,3)
\pcline[linewidth=.75 pt, linestyle=dotted]{-}(4,3)(4,4)
\pscircle[linewidth=1pt, linecolor=black, fillstyle=solid](3,3){.15}
\pscircle[linewidth=1pt, linecolor=black, fillstyle=solid](5,3){.15}
\pscircle[linewidth=1pt, linecolor=black, fillstyle=solid](4,2){.15}
\pscircle[linewidth=1pt, linecolor=black, fillstyle=solid](4,4){.15}

\pscircle[linewidth=1pt, linecolor=black, fillstyle=solid, fillcolor=black](7,3){.15}
\pcline[linewidth=.75 pt, linestyle=solid]{-}(6,3)(7,3)
\pcline[linewidth=.75 pt, linestyle=dotted]{-}(7,3)(8,3)
\pcline[linewidth=.75 pt, linestyle=dotted]{-}(7,2)(7,3)
\pcline[linewidth=.75 pt, linestyle=dotted]{-}(7,3)(7,4)
\pscircle[linewidth=1pt, linecolor=black, fillstyle=solid](6,3){.15}
\pscircle[linewidth=1pt, linecolor=black, fillstyle=solid](8,3){.15}
\pscircle[linewidth=1pt, linecolor=black, fillstyle=solid](7,2){.15}
\pscircle[linewidth=1pt, linecolor=black, fillstyle=solid](7,4){.15}

\pscircle[linewidth=1pt, linecolor=black, fillstyle=solid, fillcolor=black](10,3){.15}
\pcline[linewidth=.75 pt, linestyle=dotted]{-}(9,3)(10,3)
\pcline[linewidth=.75 pt, linestyle=solid]{-}(10,3)(11,3)
\pcline[linewidth=.75 pt, linestyle=dotted]{-}(10,2)(10,3)
\pcline[linewidth=.75 pt, linestyle=dotted]{-}(10,3)(10,4)
\pscircle[linewidth=1pt, linecolor=black, fillstyle=solid](9,3){.15}
\pscircle[linewidth=1pt, linecolor=black, fillstyle=solid](11,3){.15}
\pscircle[linewidth=1pt, linecolor=black, fillstyle=solid](10,2){.15}
\pscircle[linewidth=1pt, linecolor=black, fillstyle=solid](10,4){.15}

\pcline[linewidth=1 pt]{-}(.5,0)(1.5,1)
\pcline[linewidth=1 pt]{-}(1.5,0)(.5,1)
\pcline[linewidth=1.5 pt]{->}(1,0)(1,1)

\pcline[linewidth=1 pt]{-}(3.5,0)(4.5,1)
\pcline[linewidth=1 pt]{-}(4.5,0)(3.5,1)
\pcline[linewidth=1.5 pt]{<-}(4,0)(4,1)

\pcline[linewidth=1 pt]{-}(6.5,0)(7.5,1)
\pcline[linewidth=1 pt]{-}(7.5,0)(6.5,1)
\pcline[linewidth=1.5 pt]{<-}(6.5,.5)(7.5,.5)

\pcline[linewidth=1 pt]{-}(9.5,0)(10.5,1)
\pcline[linewidth=1 pt]{-}(10.5,0)(9.5,1)
\pcline[linewidth=1.5 pt]{->}(9.5,.5)(10.5,.5)

\end{pspicture}
	\caption{The correspondence between edges $\varepsilon$ in the overlaid Tait graph $\widehat{\Gamma}$ and directed edges $e$ in the (directed) Tait graph $G$.}
	\label{fig:CrossingTaitOverlaidTaitAGAIN}
\end{center}
\end{figure}
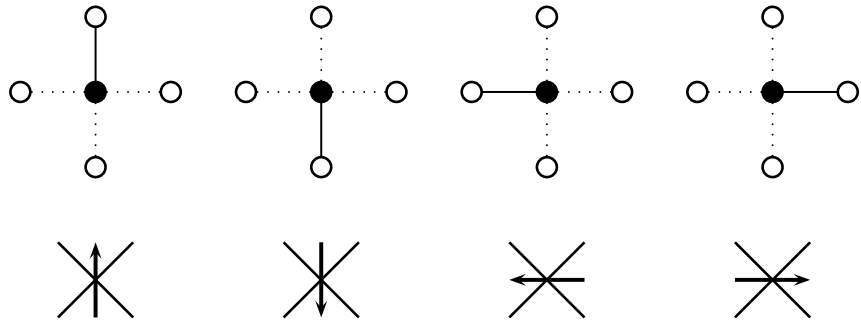

Thus it is enough to show that the directed edge $\vec{e}\in E(G)\cup E(G^*)$ belongs to some rooted spanning tree of either $G$ or its dual $G^*$.  If there are no nugatory crossings, then this edge cannot be a loop or a bridge.  So the undirected edge $e$ belongs to some un-rooted spanning tree.
%So the undirected edge $e$ belongs to some un-rooted spanning tree, and therefore one of the orientations of this edge $\vec{e}$ belongs to the corresponding rooted spanning tree.

%Thus we have left to show that the opposite orientation of this directed edge also belongs to some spanning tree.

To see that both orientations of this edge occur, it is enough to show that there is a cycle containing both $e$ and the starred vertex where $e$ is not incident to the starred vertex.  Then an edge can be removed from either end of the cycle to produce either orientation on the remaining edges of the cycle.  Observe that if $e$ is incident with the starred vertex then the edge of $\widehat{\Gamma}$ corresponding to the wrong orientation of the edge of $G$ was deleted to produce $\Gamma$.

Since $e$ is not a loop or bridge, it belongs to some cycle $C$.  Since by Proposition \ref{prop:2connG} the Tait graph $G$ is 2-connected, there can be no 1-separation, and so a separating set must be of size at least two.  Apply Menger's Theorem \ref{thm:Menger} with $A$ being the neighbors of the starred vertex and $B$ being the cycle $C$; then there must be two disjoint $A$-$B$ paths $P_1$ and $P_2$ in $G$.  In particular, the endpoints in the cycle $C$ cannot be the same, and so these endpoints partition $C$ into two paths, one of which contains $e$, say $P_3$.  Then $P_1\cup P_2\cup P_3$ is a cycle containing both $C$ and the starred vertex, as desired.
\end{proof}

In particular, this will show that $\Gamma$ for a prime-like knot diagram with no nugatory crossings is 2-connected.  First we need another possibly unfamiliar definition.

\begin{definition}
\label{extendable}
A graph $\Gamma$ is said to be \emph{$n$-extendable} if it is connected, has a set of $n$ independent lines, and
every set of $n$ independent lines in $\Gamma$ extends to (i.e. is a subset of) a perfect matching of $\Gamma$.
\end{definition}

Then by Theorem \ref{thm:elementary:LP} (5) an elementary bipartite graph is 1-extendable.

\begin{lemma}
\cite[Lemma 3.1]{Plum}
\label{lem:1ext2conn}
Every 1-extendable graph (that is not $K_2$) is 2-connected.
\end{lemma}

This gives the desired result.

\begin{proposition}
\label{prop:Gamma2conn}
The balanced overlaid Tait graph $\Gamma$ for a prime-like knot diagram with no nugatory crossings is 2-connected.
\end{proposition}

\section{Partitioning the vertex set into leaves and cycles}% $C_1,\ldots,C_k$}
\label{sec:partition}
%\section{Main Construction}
%\label{sec:MainTheorem}

%This section features the main construction of this partition, proving that at each recursive level that Property K also holds.

The main idea of this section is to partition the vertices of the graph $\Gamma=\Gamma_1$ into leaves $\ell\in L$ and cycles $C_i$, denoting by $\Gamma_i$ the  \emph{interior graph} within and including the cycle $C_i$.  The proof of the main theorem of this section, Theorem \ref{prop:concentriccircles}, will involve two induction steps:  first the existance of the next cycle and second that it satisfies the Periphery Proposition \ref{prop:PropertyK}.

These cycles $C_i$ emerge when the symmetric difference is taken of the clocked and counterclocked states $\widehat{0}$ and $\widehat{1}$ of Kauffman's clock lattice!

\bigskip

\subsection{Some final useful tools}
\label{subsec:lemmas}
  Again we make specific mention of the following definitions, as they may be unfamiliar to many readers.

%The following lemma is needed to use the definition below it.
%
%\begin{lemma}
%\label{lem:2-conn}
%\cite[Exercise 4.1.3]{LovPlum}
%If a graph is elementary bipartite and not $K_2$, then it is 2-connected.
%\end{lemma}
%
%This notion will be used in the proof of the next result:  specifically in Theorem \ref{thm:elementary:ZhaZha} used to prove Lemma \ref{lem:odd}.

\begin{definition}
\label{def:resonant}
A cycle (or a path) of $\Gamma$ is called \emph{$\mu$-alternating} if the edges of the cycle (or path) appear alternately in the perfect matching $\mu$ and $E(\Gamma)\backslash \mu$.  A face of a 2-connected plane bipartite graph is called \emph{resonant} if its boundary is a $\mu$-alternating cycle for some perfect matching $\mu$.%n alternating cycle with respect to some perfect matching of the graph.
\end{definition}

The next theorem is the reason we will be concerned with proving graphs are 2-connected and elementary.

\begin{theorem}
\label{thm:elementary:ZhaZha}
\cite[Theorem 2.4.]{ZhaZha}
Given a plane bipartite graph with more than two vertices, each face (including the infinite face)
 is resonant if and only if the graph is elementary.
%:  a plane bipartite graph $G$ is elementary if and only if the boundary of each face (including the infinite face) is an alternating cycle with respect to some perfect matching of $G$.
\end{theorem}

Denote by $q(G)$ the number of components of a graph $G$ with an odd number of vertices.  This theorem shall be used below.

\begin{theorem}
\label{thm:odd}
\cite[Theorem 2.2.1. (Tutte 1947)]{Diest}
%A graph has a perfect matching if and only if the number of odd components of the gra
A graph $G$ has a perfect matching if and only if the number of odd components $q(G-S)\leq |S|$ for all subsets $S\subseteq V(G)$.
\end{theorem}

We introduce the following two operations on cutvertices that will be used in the next subsection.

\begin{definition}
\label{def:pruning}
\textbf{``Pruning'' leaves.}  
Suppose a cutvertex is incident with a \emph{leaf}, or an edge incident with a one-valent vertex.  The operation called \emph{``pruning''} the leaf from a graph will mean deleting all edges adjacent to the leaf.  %incident with the cutvertex.  
A graph is \emph{``pruned''} when all of its leaves have been pruned.  Collect all of the leaves pruned in this way in the set $L$.  Note that here two vertices are deleted at a time, one from each vertex set if the graph is bipartite as below.
\end{definition}

\begin{definition}
\label{def:breaking}
\textbf{``Breaking'' cutvertices.}  
Suppose the deletion of a cutvertex would result in more than one component, each of which contains a cycle (when including the cutvertex).  Also suppose that there is exactly one component that has an odd number of vertices (not including the cutvertex).  The operation called \emph{``breaking''} the cutvertex from a graph will mean deleting all edges incident with the cutvertex except for those in the odd component.  A graph is \emph{``broken''} if there are no more cutvertices.
\end{definition}

\begin{remark}
These operations are non-standard in graph theory.  In other contexts, breaking a cutvertex might produce two graphs that both contain the cutvertex.  This name perhaps comes from the operation of breaking ``handcuffs'', where the tight handcuff graph is two cycles with a single vertex in common and the loose handcuff graph is two cycles connected by a path.
\end{remark}

Finally there is an argument that, while trivial, will occur in several places in the proof in the next subsection.  In particular, this will be used when a construction allows for the repeated use of new leaves rather than stopping at one of the other cases.

\begin{lemma}
\label{lem:finiteleaf}
\textbf{Finite Leaf Lemma.}
The can be only a finite number of leaves.
\end{lemma}

\begin{proof}
There is only a finite number of vertices because there is only a finite number of crossings because the knot is tame.  Wild knots are not considered in this work.
\end{proof}

\bigskip

\subsection{Main Construction}
\label{subsec:MainConstruction}

Next is the important construction:

\begin{theorem}
\label{prop:concentriccircles}
Consider the balanced overlaid Tait graph $\Gamma$ for a prime-like knot diagram with no nugatory crossings.  Then the vertices can be partitioned into leaves $\ell\in L$ and cycles $C_i$, where each cycle $C_i$ satisfies the Periphery Proposition \ref{prop:PropertyK} and where each interior graph $\Gamma_i$ is elementary and 2-connected.
\end{theorem}

By Theorem \ref{thm:elementary:ZhaZha}, these two properties of $\Gamma_i$ allow us to use induction below.
%give us the notion we will need to use below.
%tell us that, in particular, the periphery $C_i$ is resonant and that the perfect matchings of $C_i$ extend to perfect matchings in $\Gamma_i\backslash C_i$, allowing the use of induction below.

\begin{corollary}
Every face (and specifically the periphery $C_i$) of each interior graph $\Gamma_i$ is resonant.
\end{corollary}

\begin{proof}[Proof of Theorem \ref{prop:concentriccircles}]
The periphery $C=C_1$ on the infinte face satisfies the Periphery Proposition \ref{prop:PropertyK}, and $\Gamma=\Gamma_1$ is elementary by Theorem \ref{thm:elementary} and 2-connected by Proposition \ref{prop:Gamma2conn}.  The proof will proceed by constructing the next $C_{i}$ assuming that for all $j<i$ all previous $C_{j}$ already satisfy the Periphery Proposition \ref{prop:PropertyK} and that all previous $\Gamma_{j}$ are elementary and 2-connected.

Thus each previous periphery $C_j$ is resonant, and so there is a perfect matching $\mu$ such that $C_j$ is $\mu$-alternating.  This ensures that when the periphery $C_j$ is deleted from the interior graph $\Gamma_j$, the remaining graph still has some perfect matching:  specifically $\mu$ restricted to $\Gamma_j\backslash\{C_j\}$.

\smallskip

\textbf{Construction.}  Delete all the edges incident with vertices in the cycle $C_{i-1}$ from the graph $\Gamma_{i-1}$ to obtain a new graph $\Gamma_{i}'$, and consider the edges $C_{i}'$ on the new periphery.  If $C_{i}'$ has several components, treat each $C_{i}',C_{i+1}',\ldots$ separately.  If some component $C_{i}'$ is indeed a single cycle with no cutvertices, then set $C_{i}=C_{i}'$ and $\Gamma_{i}=\Gamma_{i}'$.

Otherwise there is some cutvertex, possibly on a leaf.  One can employ some sequence of ``pruning'' leaves and ``breaking'' cutvertices as follows.  It is important to note here that when there are several cutvertices, the leaves and cycles must be connected in a tree-like way, and so one can start by the outermost edges of this tree.

From pruning leaves one obtains a new set of outer edges $C_{i}''$ of the new interior graph $\Gamma_{i}''$.  If this is indeed a single cycle with no cutvertices, then set $C_{i}=C_{i}''$ and $\Gamma_{i}=\Gamma_{i}''$.  Collect the leaves pruned in this way in the set $L_{i-1}$.  Note that here two vertices are deleted at a time, preserving the property that the interior graph $\Gamma_{i}''$ has equal-sized vertex sets.

In order to break additional cutvertices, one must show that there is exactly one component that has an odd number of vertices when the cutvertex is deleted from $\Gamma_{i}''$.

\begin{lemma}
\label{lem:odd}
Suppose the graph $\Gamma_{i}'$ (or $\Gamma_{i}''$) has a cutvertex.  Then after deleting this cutvertex, exactly one connected component has an odd number of vertices.
\end{lemma}

\begin{proof}
%Since $\Gamma_{i}'$ (or $\Gamma_{i}''$) has equal-sized vertex sets, there must be at least one odd-sized component in this graph with the cutvertex deleted.
%
%To see that there is at most one odd-sized component, f
First observe by Theorem \ref{thm:elementary:ZhaZha} on the 2-connected elementary interior graph $\Gamma_{i-1}$ that there must be a perfect matching $\mu$ such that the periphery $C_{i-1}$ of $\Gamma_{i-1}$ is $\mu$-alternating.  This perfect matching $\mu$ must include the leaves deleted above.

%{\red This next paragraph doesn't make complete sense... OK, the leaves must belong to the matching.  I get it.}

%When the vertices on the periphery are deleted, it becomes clear that this perfect matching $\mu$ extends to one that includes the leaves deleted above.  

Now consider the graph $\Gamma_{i}'$ (or $\Gamma_{i}''$) with the subset $S$ being just the cutvertex as in the statement of Theorem \ref{thm:odd}.  Then the number of odd components is at most one.  Because the graph is bipartite, the number of odd components must be odd, so this number is indeed one.
%
%The perfect matching $\mu$ for $\Gamma_i$ above can be restricted to a perfect matching on $\Gamma_{i}'$ (or $\Gamma_{i}''$) because $C_{i-1}$ is $\mu$-alternating.  Then the number of odd components is indeed one.
\end{proof}

From breaking the cutvertex one obtains several components with peripheries $C_{i}''',C_{i+1}''',\ldots$ of new interior graphs $\Gamma_{i}''',\Gamma_{i+1}''',\ldots$ for each of the components.  Repeat this process until each of these is indeed a single cycle with no cutvertices, then set $C_j=C_j'''$ and $\Gamma_j=\Gamma_j'''$ for $j\geq{i}$.

\smallskip

%\textbf{Some further lemmas to finish the construction.}  The goal here is to show that 4-valent black vertices cannot be produced on $C''_i$.  From this one can deduce that whenever there is a black cutvertex on $C''_i$, there is always another white cutvertex to be broken instead.

%----------------------------------------------------------PROPERTY K-------------------------------------------------------------

%\smallskip
%
\textbf{Cycles satisfy the Periphery Proposition \ref{prop:PropertyK}. }  After a brief useful lemma, the proof continues by showing that $C_{i}$ satisfies the Periphery Proposition \ref{prop:PropertyK}.

\begin{lemma}
\label{lem:inchworm}
\textbf{Inchworm Lemma.}  
A square face $f$ in $\Gamma_{i-1}$ cannot have exactly one edge on $C_i$ while its opposite edge is a leaf $\ell\in L_{i-1}$; otherwise $C_i$ could have been extended to include $\ell$ and $f$ would be a part of $\Gamma_i$.
\end{lemma}
%{\red ``extended through $f$'' like inchworm lemma}

\begin{proof}
This follows from the construction of the cycle $C_i$ in $\Gamma_{i-1}$.
\end{proof}

%{\red Also include figure!}

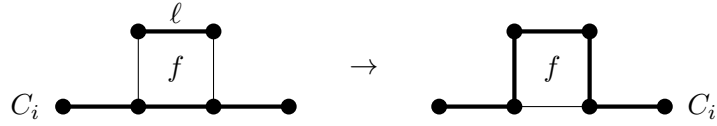
\begin{figure}[h]
\begin{center}
\begin{tikzpicture}

\draw (-.5,0) node {$C_i$};
\draw[ultra thick] (0,0) -- (3,0);
\draw (1,0) -- (1,1) -- (2,1) -- (2,0);
\draw[ultra thick] (1,1) -- (2,1);
\draw (1.5,.5) node {$f$};
\draw (1.5,1.25) node {$\ell$};

\foreach \x/ \y in {0/0, 1/0, 1/1, 2/1, 2/0, 3/0}
	{			\fill[color=black] (\x,\y) circle (3pt);				}

\draw (4,.5) node {$\rightarrow$};

\begin{scope}[shift={(5,0)}]

\draw (3.5,0) node {$C_i$};
\draw (1.5,.5) node {$f$};
%\draw (1.5,1.25) node {$\ell$};

\draw[ultra thick] (0,0) -- (1,0) -- (1,1) -- (2,1) -- (2,0) -- (3,0);
\draw (1,0) -- (2,0);

\foreach \x/ \y in {0/0, 1/0, 1/1, 2/1, 2/0, 3/0}
	{			\fill[color=black] (\x,\y) circle (3pt);				}

\end{scope}

\end{tikzpicture}
	\caption{The situation on the left cannot occur by the Inchworm Lemma \ref{lem:inchworm}.  The vertex sets are undistinguished.}
	\label{fig:inchworm}
\end{center}
\end{figure}

\begin{lemma}
\label{lem:cyclesPropertyK}
\textbf{Periphery Proposition Lemma.}  The periphery $C_i$ of the internal graph $\Gamma_i$ satisfies the Periphery Proposition \ref{prop:PropertyK}.
\end{lemma}

\begin{proof}
In order to apply the same proof of the Periphery Proposition \ref{prop:PropertyK}, we must ensure that only two-valent and three-valent black vertices can be on $C_i$.  It is clear that there are no black leaves by construction.  Thus together with Sublemma \ref{sublem:nofourvalentblackcycles} the proof can be applied in this case.
\end{proof}

\begin{sublemma}
\label{sublem:nofourvalentblackcycles}
There can be no four-valent black vertices on the cycle $C_i$.
\end{sublemma}

\begin{proof}
%First note that there can be no four-valent black cutvertices on $C_i$ because cutvertices were broken above.
%
Suppose by way of contradiction that there is a four-valent black vertex $v_1$ on the periphery $C_i$ of the interior graph $\Gamma_i$; note that $v_1$ cannot be a cutvertex as these were broken above.  Then there are two white neighbors $u_1$ and $u_2$ of $v_1$ that are also neighbors of $v_1$ on the cycle $C_i$.  Then $u_1,v_1,u_2$ must form a square face with some other black vertex $v_2$ in the interior graph $\Gamma_{i-1}$ but outside of the interior graph $\Gamma_i$.  We may assume by induction hypothesis that $v_2$ cannot be on $C_{i-1}$, or else it would be four-valent there.  Also this black vertex cannot be on $C_i$ or the cycle would extend through to include this face.

Then the black vertex $v_2$ must either be on a leaf or on another interior cycle $C_{i+1}$.  If it is on a leaf with white vertex $u_3$, then with either $u_1$ or $u_2$ this leaf forms a square face with some other black vertex $v_3$, which also cannot be on $C_{i-1}$ as above or on $C_i$ by the Inchworm Lemma \ref{lem:inchworm}.  By the Finite Leaf Lemma \ref{lem:finiteleaf}, this process eventually terminates, as so we may assume that our original other black vertex $v_2$ is on another interior cycle $C_{i+1}$.

Then $v_2$ has two white neighbors $u_3$ and $u_4$ on the cycle $C_{i+1}$.  It can have no other neighbors since a black vertex is at most four-valent.  Then $u_1$, $v_2$, and $u_3$ are on a square face in $\Gamma_{i-1}$ that also contains some black vertex $v_3$; note additionally that $u_2$, $v_2$, and $u_4$ are on a square face in $\Gamma_{i-1}$ that also contains some black vertex $v_4$, but we need only consider one of these.  Following a similar argument as above, this black vertex cannot appear on $C_{i-1}$ as it would be four-valent or $C_i$ or $C_{i+1}$ as this would alter the cycle structure.  It may appear on a leaf, but by the Finite Leaf Lemma \ref{lem:finiteleaf} this proceuss eventually terminates.  The only remaining option is for $v_3$ to appear on a new interior cycle $C_{i+2}$.

However, this situation again forces two new black vertices that must be handled according to the above arguments.  What is more is that no new black vertex created can appear on a previous interior cycle, as it would alter the cycle structure.  Thus ultimately all new interior cycles and leaves are exhausted and there are no remaining options for the black vertex, a contradiction.
\end{proof}

\smallskip

\textbf{$\Gamma_i$ is 2-connected and elementary.}  The interior graph $\Gamma_{i}$ is 2-connected by construction: it is connected because each component was considered separately, and all cutvertices were removed after ``pruning'' leaves and ``breaking'' cutvertices.

To prove that the interior graph $\Gamma_{i}$ is elementary, we apply the proof of Theorem \ref{thm:elementary} to the interior graph $\Gamma_{i}$ after it has been turned into a diagram $D_{i}$ by Proposition \ref{prop:UniqueUniverse}.  Note that this diagram is prime-like with no nugatory crossings by construction.  It does not matter if the diagram represents a knot or a link with several components.
\end{proof}

\begin{remark}
As inferred by the end of the proof, the interior graphs $\Gamma_i$ are in fact themselves balanced overlaid Tait graphs.
\end{remark}

\begin{question}
Can the balanced overlaid Tait graphs $\Gamma_i$ be related to each other in an analogous way to tangles and sub-tangles?
\end{question}

\section{Reduction moves and black two-valent vertices}
\label{sec:2valent}

The purpose of this section is to highlight several local moves that simplify the structure of the interior graph $\Gamma_{i-1}$.  These will be used to prove the main result, Theorem \ref{conj:sum}, on the height of Kauffman's clock lattice.

The first reduction move can be used to simplify each interior graph $\Gamma_{i-1}$ by removing leaves without altering the structure of the cycles.

\begin{proposition}
\label{prop:leafreduction}
\textbf{Leaf Reduction Proposition.}
Suppose the interior graph $\Gamma_{i-1}$ obtained from a knot diagram that has no nugatory crossings contains some alternating path of leaves.  Then this graph can be reduced to one without the leaves by a local move as depicted in Figure \ref{fig:leafreduction}.
\end{proposition}

\begin{figure}[h]
\begin{center}
\begin{tikzpicture}

%\draw (-1,1.5) node {$(A)$};

\draw[ultra thick] (2,1) -- (2,2);
\draw (0,0) -- (4,0) -- (4,3) -- (2,2) -- (0,3) -- (0,0) -- (2,1) -- (4,0) -- cycle;
\draw (2,0) -- (2,3);
\draw (5,1.5) node {$\rightarrow$};
\draw (6,0) -- (10,0) -- (10,3) -- (8,0) -- (6,3) -- cycle;
\draw (8,0) -- (8,3);
\foreach \x/ \y in {0/0, 2/0, 4/0, 2/2, 6/0, 8/0, 10/0}
	{			\fill[color=white] (\x,\y) circle (3pt);
				\draw (\x,\y) circle (3pt);											}
\foreach \x/ \y in {1/0, 3/0, 2/1, 0/3, 2/3, 4/3, 7/0, 9/0, 6/3, 8/3, 10/3}
	{			\fill[color=black] (\x,\y) circle (3pt);				}
\draw (1,3) node {$\cdots$};
\draw (3,3) node {$\cdots$};
\draw (7,3) node {$\cdots$};
\draw (9,3) node {$\cdots$};

\draw (2.5,1.25) node {$v_1$};
\draw (2.5,1.75) node {$u_1$};
\draw (0,-.5) node {$u_2$};
\draw (1,-.5) node {$v_2$};
\draw (2,-.5) node {$u_3$};
\draw (3,-.5) node {$v_3$};
\draw (4,-.5) node {$u_4$};

\draw (0,3.5) node {$v_4$};
\draw (4,3.5) node {$v_5$};

\end{tikzpicture}
	\caption{The local operation of reduction on a single leaf.}
	\label{fig:leafreduction}
\end{center}
\end{figure}
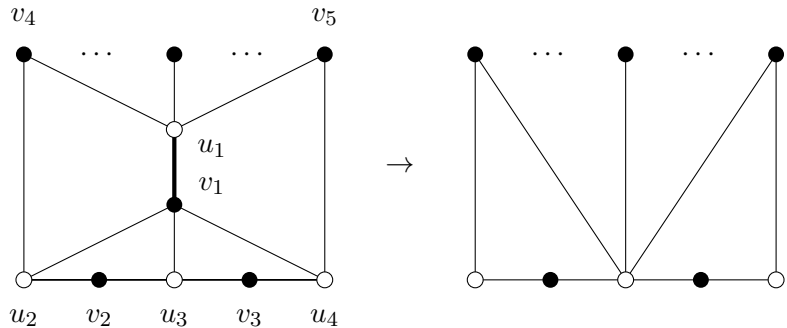

\begin{proof}
It is enough to show that the square faces depicted in the left hand side of Figure \ref{fig:leafreduction} arise.  Then they can be reduced to those that appear on the right hand side because the graph is still bipartite, all faces are still square, and the valence of the black vertices does not change.

Consider the black vertex $v_1$ on the leaf along with its white neighbors $u_1$, $u_2$, $u_3$, and $u_4$ as in the left hand side of Figure \ref{fig:leafreduction}.  Observe that these four neighbors must be distinct:  identifying any two forms either a bigon or a nugatory crossing that would be in both $\Gamma_{i-1}$ as well as $\Gamma$.  Then there is a square face containing $u_2$, $v_1$, and $u_3$ that must also contain some black vertex $v_2$ and a square face containing $u_3$, $v_1$, and $u_4$ that must also contain some black vertex $v_3$.

If $v_2=v_3$ are not distinct, this produces another leaf as in Figure \ref{fig:leafreductionB}, and there are again two black vertices to be considered.  Continue this process creating the alternating path of leaves, ending with two distinct black vertices because of the Finite Leaf Lemma \ref{lem:finiteleaf}.  This creates the square faces depicted in the lower parts of the images on the left hand side of Figures \ref{fig:leafreduction} and \ref{fig:leafreductionB}.

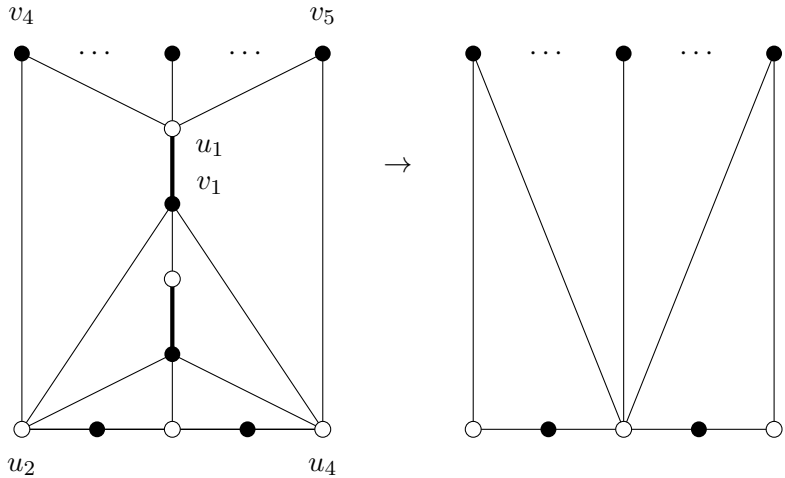
\begin{figure}[h]
\begin{center}
\begin{tikzpicture}

%\draw[shift={(0,-5)}] (-1,1.5) node {$(B)$};

\draw[shift={(0,-5)}, ultra thick] (2,1) -- (2,2);
\draw[shift={(0,-5)}, ultra thick] (2,-1) -- (2,0);
\draw[shift={(0,-5)}] (0,-2) -- (4,-2) -- (4,3) -- (2,2) -- (0,3) -- (0,-2) -- (2,1) -- (4,-2) -- cycle;
\draw[shift={(0,-5)}] (2,-2) -- (2,3);
\draw[shift={(0,-5)}] (5,1.5) node {$\rightarrow$};
\draw[shift={(0,-5)}] (6,-2) -- (10,-2) -- (10,3) -- (8,-2) -- (6,3) -- cycle;
\draw[shift={(0,-5)}] (8,-2) -- (8,3);
\draw[shift={(0,-5)}] (0,-2) -- (2,-1) -- (4,-2);
\foreach \x/ \y in {0/-2, 2/-2, 4/-2, 2/0, 2/2, 6/-2, 8/-2, 10/-2}
	{			\fill[shift={(0,-5)}, color=white] (\x,\y) circle (3pt);
				\draw[shift={(0,-5)}] (\x,\y) circle (3pt);											}
\foreach \x/ \y in {1/-2, 3/-2, 2/-1, 2/1, 0/3, 2/3, 4/3, 7/-2, 9/-2, 6/3, 8/3, 10/3}
	{			\fill[shift={(0,-5)}, color=black] (\x,\y) circle (3pt);				}
\draw[shift={(0,-5)}] (1,3) node {$\cdots$};
\draw[shift={(0,-5)}] (3,3) node {$\cdots$};
\draw[shift={(0,-5)}] (7,3) node {$\cdots$};
\draw[shift={(0,-5)}] (9,3) node {$\cdots$};

\draw[shift={(0,-5)}] (2.5,1.25) node {$v_1$};
\draw[shift={(0,-5)}] (2.5,1.75) node {$u_1$};
\draw[shift={(0,-7)}] (0,-.5) node {$u_2$};
\draw[shift={(0,-7)}] (4,-.5) node {$u_4$};

\draw[shift={(0,-5)}] (0,3.5) node {$v_4$};
\draw[shift={(0,-5)}] (4,3.5) node {$v_5$};

%\draw[ultra thick] (1,2) -- (2,2);
%\draw (0,0) -- (3,0) -- (2,2) -- (3,4) -- (0,4) -- (0,0) -- (1,2) -- (0,4) -- cycle;
%\draw (0,2) -- (3,2);
%\draw (4,2) node {$\rightarrow$};
%\draw (5,0) -- (8,0) -- (5,2) -- (8,4) -- (5,4) -- cycle;
%\draw (5,2) -- (8,2);
%\foreach \x/ \y in {0/0, 0/2, 0/4, 2/2, 5/0, 5/2, 5/4}
%	{			\fill[color=white] (\x,\y) circle (3pt);
%				\draw (\x,\y) circle (3pt);											}
%\foreach \x/ \y in {0/1, 0/3, 1/2, 3/0, 3/2, 3/4, 5/1, 5/3, 8/0, 8/2, 8/4}
%	{			\fill[color=black] (\x,\y) circle (3pt);				}
%\draw (3,1) node {$\vdots$};
%\draw (3,3) node {$\vdots$};
%\draw (8,1) node {$\vdots$};
%\draw (8,3) node {$\vdots$};
	
\end{tikzpicture}
	\caption{The local operation of reductions on a path of leaves.}
	\label{fig:leafreductionB}
\end{center}
\end{figure}

To see the square faces in the upper parts of these images, observe that there is a square face containing $u_1$, $v_1$, and $u_2$ that must also contain some black vertex $v_4$ and a square face containing $u_1$, $v_1$, and $u_4$ that must also contain some black vertex $v_5$.  Even if $v_4=v_5$ are not distinct, this move can be performed.
\end{proof}

\begin{remark}
A generalized version of this reduction move can apply to leaves alternating on some tree by starting at one-valent white vertices on this tree, as in Figure \ref{fig:leafreductionB}.
\end{remark}

\begin{remark}
This Leaf Reduction Proposition move is in fact a number of smoothings on a twist region of the knot diagram, where the number of crossings in the twist region is given by the number of leaves in the path.
\end{remark}

\begin{proposition}
\label{prop:unstacking}
\textbf{Simply Connected Region Reduction Proposition.}
%\textbf{Unstacking Proposition.}
Let $H$ be an induced subgraph of the interior graph $\Gamma_{i-1}$ with all its vertices on $C_{i-1}$ such that all but one of the edges of the periphery of $H$ are on $C_{i-1}$.  %the periphery $C_{i-1}$ of $\Gamma_{i-1}$.  
Then there is a two-valent black vertex of $C_{i-1}\subset\Gamma_{i-1}$ in $H$, and $H$ can be deleted from $\Gamma_{i-1}$ by a local move without changing the structure of interior cycles. %as depicted in Figure {\red ???}.
%Suppose the interior graph $\Gamma_{i-1}$ obtained from a knot diagram that has no nugatory crossings contains a simply connected induced subgraph with all its vertices on $C_{i-1}$.  Then this graph can be reduced to one without this region by a local move as depicted in Figure .
\end{proposition}

%{\red
%Probably no FIGURE???
%}

\begin{proof}
Consider the two vertices $v_1$ and $u_1$ on the edge $e$ not on $C_{i-1}$.  By the Periphery Proposition Lemma \ref{lem:cyclesPropertyK} the black vertex $v_1$ must be three-valent on $C_{i-1}$ and so has a white neighbor $u_2$ on $H$.  Then there is a square face in $H$ containing $u_1$, $v_1$, $u_2$, and some black vertex $v_2$, which must also be on $C_{i-1}$, as a neighbor of either $u_1$ or $u_2$.  Without loss of generality we may say it is $u_1$; then there is an edge $u_2v_2$ of this square face that encloses a new induced subgraph $H'$ together with a path on $C_{i-1}$ that has one fewer square face.

In this way a simply connected induced subgraph $H$ can be ``unstacked'' until there is only one face left with all four vertices on $C_{i-1}$.  It is clear then that this gives a two-valent black vertex on $\Gamma_{i-1}$.  When all faces of $H$ have been deleted, the original black vertex $v_1$ becomes two-valent in the new $\Gamma_{i-1}$.
%In this way the induced subgraph $H$ can be ``unstacked'' until there is only one face left with all four vertices on $C_{i-1}$.  It is clear then that this gives a two-valent black vertex on $\Gamma_{i-1}$.  When all faces of $H$ have been deleted, the original black vertex $v_1$ becomes 2-valent in $\Gamma_{i-1}$.
\end{proof}

\begin{remark}
\label{rem:globloc}
While the Leaf Reduction Proposition \ref{prop:leafreduction} works globally, one cannot employ the Simply Connected Region Reduction %{\red Unstacking} 
Proposition \ref{prop:unstacking} in $\Gamma_{i-1}$ while considering the larger interior graph $\Gamma_{i-2}$.
\end{remark}

%{\red NEW TRANSITION?}

Since Theorem \ref{prop:concentriccircles} gives two black two-valent vertices on each cycle $C_i$, we will investigate further these black vertices.

%Let us now consider the black two-valent vertices on all of the cycles; in particular, such a vertex on $C_{i-1}$ gives way explicitly to another on $C_i$.

\begin{proposition}
\label{prop:2to2}
\textbf{Stacking Proposition.}
A black two-valent vertex on $\Gamma_{i-1}$ produces one on $\Gamma_i$.
%Suppose $\Gamma$ can be decomposed into cycles $C_1,\ldots,C_k$ and leaves as above.  Then a black two-valent vertex on $C_{i-1}$ produces one on $C_i$ for $i\leq k$.
\end{proposition}

\begin{proof}
First use the Leaf Reduction Proposition \ref{prop:leafreduction} to remove any leaves.

Now consider one of the two black two-valent vertices $v_1$ in $C_{i-1}$ and its two white neighboring vertices $u_1,u_2$ also in $C_{i-1}$.  Then the square face on the interior graph $\Gamma_{i-1}$ containing $u_1$, $v_1$, and $u_2$ has one additional black vertex $v_2$.

If $v_2$ is not in $C_{i-1}$, then it is a two-valent vertex in some interior cycle $C_i$ as in Figure \ref{fig:two-valent} (A).

\begin{figure}[h]
\begin{center}
\begin{tikzpicture}

\draw (-.5,2.5) node {$(A)$};

	\draw[ultra thick] (0,3) -- (6,3);
\draw (6.5,3) node {$C_{i-1}$};
	\draw[ultra thick] (0,2) -- (6,2);
\draw (6.5,2) node {$C_{i}$};

	\draw (1,3) -- (2,2) -- (3,3);
	
	\fill[color=white] (1,3) circle (3pt);
	\draw (1,3) circle (3pt);
\draw (1,3.25) node {$u_1$};
	\fill[color=black] (2,3) circle (3pt);
\draw (2,3.25) node {$v_1$};
	\fill[color=white] (3,3) circle (3pt);
	\draw (3,3) circle (3pt);
\draw (3,3.25) node {$u_2$};
	\fill[color=black] (2,2) circle (3pt);
\draw (2,1.75) node {$v_2$};

\draw (-.5,-.25) node {$(B)$};

	\draw[ultra thick] (0,.5) -- (6,.5);
\draw (6.5,.5) node {$C_{i-1}$};
	\draw[ultra thick] (0,-1) -- (6,-1);
\draw (6.5,-1) node {$C_{i}$};

	\draw (1,.5) .. controls (2,-.25) and (3,-.25) .. (4,.5);

	\draw (1,.5) -- (3,-1) -- (5,.5);
	
	\fill[color=white] (1,.5) circle (3pt);
	\draw (1,.5) circle (3pt);
\draw (1,.75) node {$u_1$};
	\fill[color=black] (2,.5) circle (3pt);
\draw (2,.75) node {$v_1$};
	\fill[color=white] (3,.5) circle (3pt);
	\draw (3,.5) circle (3pt);
\draw (3,.75) node {$u_2$};
	\fill[color=black] (4,.5) circle (3pt);
\draw (4,.75) node {$v_2$};
	\fill[color=white] (5,.5) circle (3pt);
	\draw (5,.5) circle (3pt);
\draw (5,.75) node {$u_3$};
	\fill[color=black] (3,-1) circle (3pt);
\draw (3,-1.25) node {$v_3$};

\end{tikzpicture}
	\caption{The two-valent vertex $v_1$ in $C_{i-1}$ producing another two-valent vertex in $C_i$.}
	\label{fig:two-valent}
\end{center}
\end{figure}
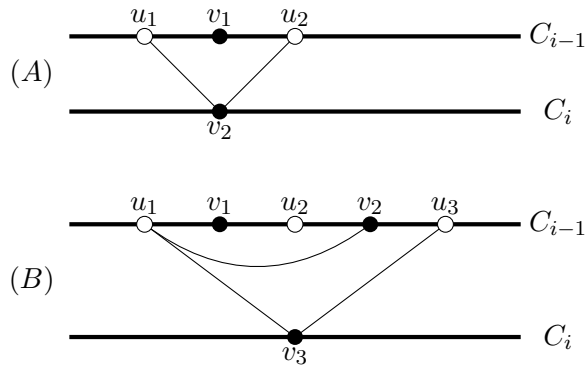

If $v_2$ is indeed in $C_{i-1}$, then at least one of the two edges to $v_2$ from the $u_1$, $u_2$ must be in $C_{i-1}$, as well, in order to keep $v_2$ at most three-valent in $\Gamma_{i-1}$.  When both of these edges belong to $C_{i-1}$, the graph $\Gamma_{i-1}=\Gamma_k$ is just a square, and there is no $C_i$.  Supposing the edge $u_1v_2$ is not in $C_{i-1}$, we obtain the scenario depicted at the top of Figure \ref{fig:two-valent} (B).

In this way more square faces may be stacked on top of each other, each pair meeting at a single edge, with all vertices appearing on $C_{i-1}$.%  This figure only shows one complete square face with all of its vertices in $C_{i-1}$, but in general it is possible to see several squares stacked on top of each other with all vertices  in $C_{i-1}$.

The outer two white vertices surrounding all the three-valent black vertices must share a common neighbor, say $v_3$.  If $v_3$ is in $C_{i-1}$, as well, it is the other two-valent vertex in $C_{i-1}$, the graph $\Gamma_{i-1}$ is just these stacked squares, and there is no $C_i$.  Otherwise $v_3$ is a two-valent vertex in $C_i$ as in the lower portion of Figure \ref{fig:two-valent} (B).

Thus one can see directly how two black two-valent vertices on interior cycles arise from those on $C_{i-1}$.
\end{proof}

\begin{remark}
It may happen, however, that the two black two-valent vertices produced in the proof above occur on two separate cycles $C_i$ and $C_{i+1}$.
\end{remark}

Following this last remark, additional single black two-valent vertices arise when more than one connected component appears after deleting the periphery.  This is handled by the following proposition.

\begin{proposition}%[Accordion Proposition]
\label{prop:accordion}
\textbf{Accordion Proposition.}
Suppose that when $C_{i-1}$ is deleted from $\Gamma_{i-1}$ it results in two components, each containing cycles after pruning leaves and breaking cutvertices.  Let $C_i$ and $C_{i+1}$ be the two cycles, one in each component, that are closest to each other using the usual notion of distance on the graph.  Then this results in an extra pair of black two-valent vertices, one on each of $C_i$ and $C_{i+1}$.
\end{proposition}

As in Figure \ref{fig:goodaccordion}, the two interior cycles $C_i$ and $C_{i+1}$ appear as handles of the so-called \emph{accordion}.

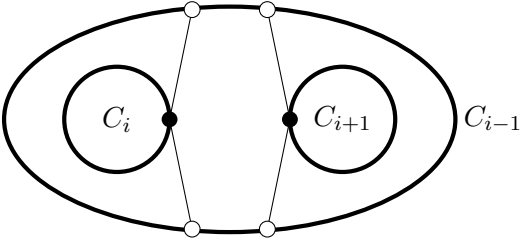
\begin{figure}[h]
\begin{center}
\begin{tikzpicture}

%\begin{scope}[scale=.5]
\draw (3.5,0) node {$C_{i-1}$};
\draw[ultra thick] (0,0) ellipse (3cm and 1.5cm);
	\draw[ultra thick] (1.5,0) circle (.7cm);
	\draw[ultra thick] (-1.5,0) circle (.7cm);
\draw (1.5,0) node {$C_{i+1}$};
\draw (-1.5,0) node {$C_{i}$};
	\draw (.8,0) -- (.5,1.46);
	\draw (.8,0) -- (.5,-1.46);
	\draw (-.8,0) -- (-.5,1.46);
	\draw (-.8,0) -- (-.5,-1.46);

%	\fill[color=black] (.5,1.46) circle (3pt);
%	\fill[color=black] (.5,-1.46) circle (3pt);
%	\fill[color=black] (-.5,1.46) circle (3pt);
%	\fill[color=black] (-.5,-1.46) circle (3pt);
	\fill[color=white] (.5,1.46) circle (3pt);
	\draw (.5,1.46) circle (3pt);
	\fill[color=white] (-.5,1.46) circle (3pt);
	\draw (-.5,1.46) circle (3pt);
	\fill[color=white] (.5,-1.46) circle (3pt);
	\draw (.5,-1.46) circle (3pt);
	\fill[color=white] (-.5,-1.46) circle (3pt);
	\draw (-.5,-1.46) circle (3pt);
	\fill[color=black] (.8,0) circle (3pt);
	\fill[color=black] (-.8,0) circle (3pt);
%\end{scope}
	
\end{tikzpicture}
	\caption{An accordion arises when $C_{i-1}$ is deleted from $\Gamma_{i-1}$ leaving two components.}
	\label{fig:goodaccordion}
\end{center}
\end{figure}

%{\red By the Stacking Proposition \ref{prop:2to2} we may assume something, maybe...}

\begin{proof}
First use the Leaf Reduction Proposition \ref{prop:leafreduction} to remove the leaves. Observe that this does not change the structure of the cycles in consideration here.

Because the interior graph $\Gamma_{i-1}$ is 2-connected, Menger's Theorem \ref{thm:Menger} states that there are at least two internally disjoint paths from $C_i$ to $C_{i+1}$.  Specifically, one can choose two disjoint paths whose interior vertices are all on $C_{i-1}$, one on each side.

Choose two such paths, each with minimal distance.  Observe that this distance must be at least two; otherwise $C_i$ and $C_{i+1}$ are connected by an edge in $\Gamma_{i-1}\backslash C_{i-1}$, contradicting the assumption.

These two paths form a cycle that is the periphery of an induced subgraph $H$ with no other vertices in it, since the Leaf Reduction Proposition has removed leaves and since $C_i$ and $C_{i+1}$ are nearest to each other by distance.  The periphery of $H$ is composed of some possibly trivial path on $C_{i-1}$, an edge from $C_{i-1}$ to $C_i$, some possibly trivial path on $C_i$, another edge from $C_{i}$ to $C_{i-1}$, another possibly trivial path on $C_{i-1}$, an edge from $C_{i-1}$ to $C_{i+1}$, some possibly trivial path on $C_{i+1}$, and another edge from $C_{i+1}$ to $C_{i-1}$ as in Figure \ref{fig:badaccordion}.% {\red ???}

%{\red FIGURE ???}

\begin{figure}[h]
\begin{center}
\begin{tikzpicture}

%\begin{scope}[scale=.5]
\draw (0,0) node {$H$};
\draw (3.5,0) node {$C_{i-1}$};
\draw[ultra thick] (0,0) ellipse (3cm and 1.5cm);
	\draw[ultra thick] (1.5,0) circle (.7cm);
	\draw[ultra thick] (-1.5,0) circle (.7cm);
\draw (1.5,0) node {$C_{i+1}$};
\draw (-1.5,0) node {$C_{i}$};
	\draw (1,.5) -- (.5,1.46);
	\draw (1,-.5) -- (.5,-1.46);
	\draw (-1,.5) -- (-.5,1.46);
	\draw (-1,-.5) -- (-.5,-1.46);

	\fill[color=black] (.5,1.46) circle (3pt);
	\fill[color=black] (.5,-1.46) circle (3pt);
	\fill[color=black] (-.5,1.46) circle (3pt);
	\fill[color=black] (-.5,-1.46) circle (3pt);
	\fill[color=black] (1,.5) circle (3pt);
	\fill[color=black] (-1,.5) circle (3pt);
	\fill[color=black] (-1,-.5) circle (3pt);
	\fill[color=black] (1,-.5) circle (3pt);
%\end{scope}
	
\end{tikzpicture}
	\caption{A general set-up used in the proof of the Accordion Proposition \ref{prop:accordion}, with the vertex sets undistinguished.}
	\label{fig:badaccordion}
\end{center}
\end{figure}
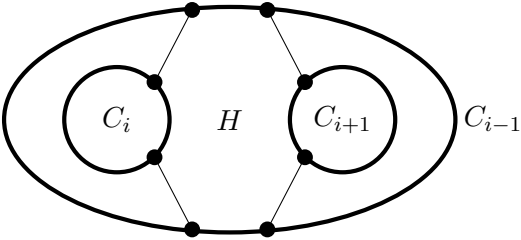

Next we show that the periphery of the induced subgraph $H$ satisfies the Periphery Proposition \ref{prop:PropertyK}.  As in the Periphery Proposition Lemma \ref{lem:cyclesPropertyK}, in order to apply the same proof of the Periphery Proposition \ref{prop:PropertyK}, we must ensure that only two-valent and three-valent black vertices can be on the periphery of $H$.  It is clear that there are no black leaves by construction.  Thus together with Lemma \ref{lem:accordionnofourvalentblackcycles} the proof can be applied in this case.

\begin{lemma}
\label{lem:accordionnofourvalentblackcycles}
There can be no four-valent black vertices on the periphery of $H$.
\end{lemma}

\begin{proof}
Such a vertex could not be on $C_{i-1}$ as it would be four-valent there, contradicting Theorem \ref{prop:concentriccircles}, so it would have to be on $C_i$ or $C_{i+1}$.  However, since there are no vertices in $H$ except on its periphery, this would contradict the minimal distance path assumption.
\end{proof}

Thus there are two black two-valent vertices.  We will show that one of these sits on $C_i$ and the other on $C_{i+1}$ in such a way that they are also two-valent there.

Suppose by way of contradiction one of these, say $v_1$, appears on $C_{i-1}$; let $u_1$ and $u_2$ be its two white neighbors there.  Then there is a square face in $H$ containing $u_1$, $v_1$, $u_2$, and some black vertex $v_2$, which must lie on the periphery of $H$ because there are no interior vertices.  However, this vertex cannot appear on $C_{i-1}$ as it would either be four-valent there or it would contradict the minimality assumption, and it cannot appear on $C_i$ or $C_{i+1}$ as it would contradict the minimality assumption there, as well.

Then the two-valent black vertex $v_1$ on the periphery of $H$ must appear on paths on either $C_i$ or $C_{i+1}$, say $C_i$.  However, it cannot appear as an internal vertex on either of these paths, as there would be a square face in $H$ containing it together with its two white neighbors as well as some other black vertex $v_2$.  This vertex cannot lie on $C_{i-1}$ or $C_i$ as it would contradict the minimality assumption, and it cannot lie on $C_{i+1}$ as it would contradict the assumption that $C_i$ and $C_{i+1}$ are not connected by an edge.

Not only must $v_1$ appear as an endpoint of the path on $C_i$, but it must appear as the trivial path itself on $C_i$; otherwise there is a square face in $H$ containing it together with its white neighbors as well as some other black vertex $v_2$.  This vertex cannot lie on $C_{i-1}$ as it would either contradict the minimality assumption or be four-valent on $C_{i-1}$, it cannot lie on $C_i$ as it would contradict the minimality assumption, and it cannot lie on $C_{i+1}$ as it would contradict the assumption that $C_i$ and $C_{i+1}$ are not connected by an edge.

Thus the path $C_i$ is in fact trivial, and $v_1$ is the only vertex from the periphery of $H$ that is on it.  Since it is two-valent in $H$ and in the interior of the interior graph $\Gamma_{i-1}$, it must be four-valent in $\Gamma_{i-1}$ with its other two edges as edges of $C_i$.  Thus it is two-valent in $\Gamma_i$.

The second two-valent black vertex on the periphery of $H$ can then only appear as a two-valent black vertex on $C_{i+1}$ following the procedure above, completing the proof.
\end{proof}

\begin{remark}
Consider the number of faces in $H$.  When $H$ is just a single face, the two black vertices on $C_i$ and $C_{i+1}$ can be thought of as handles.  When $H$ contains more faces, these handles can be thought of as ``stretched'' like an accordion to enlarge the paths on $C_{i-1}$.

Similarly, such a region $H$ can be ``compressed'' like an accordion into a single face by a local move imitating the Simply Connected Region Reduction Proposition \ref{prop:unstacking} without altering the structure of the cycles within $\Gamma_{i-1}$.   This cannot be employed globally as it may effect $\Gamma_{i-2}$.
\end{remark}

\begin{remark}
If $\Gamma_{i-1}\backslash C_{i-1}$ has more than two connected components, they must be arranged in a tree-like fashion and so the Accordion Proposition \ref{prop:accordion} can be applied to the pair of components along each edge of this tree. 
\end{remark}

There is one last way for several interior cycles to appear, and that is from breaking cutvertices.  We show how new two-valent black vertices arise in this context.

%Suppose that when $C_{i-1}$ is deleted from $\Gamma_{i-1}$ it results in two components, each containing cycles after pruning leaves and breaking cutvertices.  Let $C_i$ and $C_{i+1}$ be the two cycles, one in each component, that are closest to each other using the usual notion of distance on the graph.  Then this results in an extra pair of black two-valent vertices, one on each of $C_i$ and $C_{i+1}$.

\begin{proposition}
\label{prop:2b2break}
\textbf{Party Hat Proposition.}
Suppose that when a white cutvertex $u_1$ is broken in the interior graph $\Gamma_{i-1}$ it results in two components, each containing cycles after pruning leaves and breaking additional cutvertices.  Let $C_i$ and $C_{i+1}$ be the two cycles, one in each component with $C_i$ containing the original cutvertex, that are closest to each other using the usual notion of distance on the graph.  Then this results in an extra pair of black two-valent vertices, both on $C_{i+1}$.
%Suppose that a white cutvertex was broken in the interior graph $\Gamma_{i-1}$ directly resulting in several cycles $C_i,\ldots,C_{i+j}$ that were attached to the cutvertex (possibly by leaves).  Then this results in an additional $j$ pairs of black two-valent vertices amongst the cycles $C_i,\ldots,C_{i+j}$, each pair to a cycle.
%If breaking a white cutvertex in the interior graph $\Gamma_{i-1}'$ results in several components $C_i',\ldots,C_{i+j}'$ each with at least one cycle, then this results in an extra $j$ pairs of black two-valent vertices amongst the components $C_i',\ldots,C_{i+j}'$, each pair to a cycle.
%%Breaking a white cutvertex into two $j+1$ results in $2j$ additional black two-valent vertices on the interior cycles.%some interior cycle $C_i$.
\end{proposition}

As in Figure \ref{fig:goodpartyhat}, the interior cycle $C_{i}$ appears as the puff ball at the top of the so-called \emph{party hat} worn by the interior cycle $C_{i+1}$.

\begin{figure}[h]
\begin{center}
\begin{tikzpicture}

%\begin{scope}[scale=.5]
\draw (3.5,0) node {$C_{i-1}$};
\draw[ultra thick] (0,0) ellipse (3cm and 1.5cm);
	\draw[ultra thick] (1.5,0) circle (.7cm);
	\draw[ultra thick] (-1.5,0) circle (.7cm);
\draw (1.5,0) node {$C_{i+1}$};
\draw (-1.5,0) node {$C_{i}$};
	\draw (.5,1.46) -- (1,.5) -- (-.8,0) -- (1,-.5) -- (.5,-1.46);
%	\draw (.8,0) -- (.5,-1.46);
%	\draw (-.8,0) -- (-.5,1.46);
%	\draw (-.8,0) -- (-.5,-1.46);

	\fill[color=white] (.5,1.46) circle (3pt);
	\draw (.5,1.46) circle (3pt);
	\fill[color=white] (.5,-1.46) circle (3pt);
	\draw (.5,-1.46) circle (3pt);
%	\fill[color=black] (-.5,1.46) circle (3pt);
%	\fill[color=black] (-.5,-1.46) circle (3pt);
%	\fill[color=black] (.8,0) circle (3pt);
	\fill[color=white] (-.8,0) circle (3pt);
	\draw (-.8,0) circle (3pt);
\draw (-.5,-.5) node {$u_1$};
	\fill[color=black] (1,.5) circle (3pt);
%	\draw (1,.5) circle (3pt);
	\fill[color=black] (1,-.5) circle (3pt);
%	\draw (1,-.5) circle (3pt);
%\end{scope}
	
\end{tikzpicture}
	\caption{A party hat arises if a cutvertex is broken when $C_{i-1}$ is deleted from $\Gamma_{i-1}$.}
	\label{fig:goodpartyhat}
\end{center}
\end{figure}
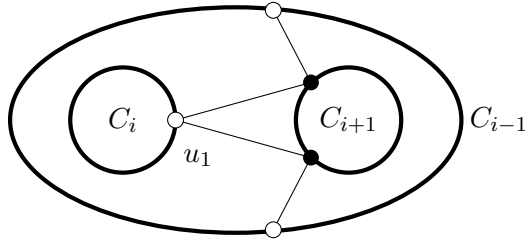

\begin{proof}
First use the Leaf Reduction Proposition \ref{prop:leafreduction} to remove the leaves.  Observe that this does not change the structure of the cycles in consideration here.

Recall from Lemma \ref{lem:odd} that after deleting the cutvertex $u_1$ exactly one of the components has an odd number of vertices; say this is cycle $C_i$.  Look at the black neighbors of $u_1$ that are on $C_{i+1}$ and consider their cyclic order around $u_1$.  

\begin{lemma}
The cutvertex $u_1$ has valence at least two in the induced subgraph together with the cycle $C_{i+1}$.
\end{lemma}

\begin{proof}
Suppose by way of contradiction that $u_1$ has valence one in $C_{i+1}\cup\{u_1\}$; let $v_1$ be its neighbor on $C_{i+1}$.  Since $v_1$ is four-valent in $\Gamma_{i-1}$, it has one extra neighbor besides $u_1$ and the two on $C_{i+1}$.  If this is outside of $\Gamma_{i+1}$ consider the neighbor $u_2$ of $v_1$ on the opposite side of $u_1$; otherwise choose either neighbor.

Then the square face outside the interior graph $\Gamma_{i+1}$ containing $u_1$, $v_1$, and $u_2$ has one additional black vertex $v_2$.  This black vertex cannot be on $C_{i-1}$ as it would be four-valent there; it cannot be on $C_i$ as it would give way to a larger cycle containing both $C_i$ and $C_{i+1}$.  If it is on $C_{i+1}$, this contradicts the assumption that $u_1$ has valence one there.

%-- possibly copy the paragraph about $C_{i+2}$.

There are no leaves to consider by the Leaf Reduction Proposition \ref{prop:leafreduction}, so $v_2$ must be on some other cycle $C_{i+2}$, where it is two-valent.  %Consider its white neighbor $u_3$ on $C_{i+2}$ on the same side as $u_1$.  Then the square face containing $u_1$, $v_2$, and $u_3$ has one additional black vertex $v_3$.  %NOT TRUE!  As above this black vertex cannot be on $C_{i-1}$, $C_i$, or $C_{i+2}$.

Then following the proof of Sublemma \ref{sublem:nofourvalentblackcycles}, this situation again forces two new black vertices that must be handled according to the above arguments.  What is more is that no new black vertex created can appear on a previous interior cycle, as it would alter the cycle structure.  Thus ultimately all new interior cycles and leaves are exhausted, and there are no remaining options for the black vertex, a contradiction.
\end{proof}

By the lemma above, there are at least two black neighbors of $u_1$ on $C_{i+1}$, and so we may consider just the two on the outside (on either side) of this cyclic order; call these $v_1$ and $v_2$.

We show that $v_1$ and $v_2$ are two-valent in $C_{i+1}$ and thus there are no other two-valent black vertices on $C_{i+1}$ by the Periphery Proposition Lemma \ref{lem:cyclesPropertyK}.

If not $v_1$ must be three-valent following the lemma.  Then there is a square face in $\Gamma_{i-1}$ outside of $\Gamma_i$ and $\Gamma_{i+1}$ containing $u_1$, $v_1$, and a white neighbor of $v_1$ on $C_{i+1}$ that also contains some black vertex $v_3$.

This black vertex $v_3$ cannot be on $C_{i-1}$ as it would be four-valent there; it cannot be on $C_i$ or $C_{i+1}$ as this would change the cycle structure.  Since there are no leaves, it must be on some new $C_{i+2}$.

Then following the proof of Sublemma \ref{sublem:nofourvalentblackcycles}, this situation again forces two new black vertices that must be handled according to the above arguments.  What is more is that no new black vertex created can appear on a previous interior cycle, as it would alter the cycle structure.  Thus ultimately all new interior cycles and leaves are exhausted and there are no remaining options for the black vertex, a contradiction.
\end{proof}

%{\red this is old... include?}
\begin{remark}
Note that the two new black two-valent vertices appear on the same component $C_{i+1}$.  Since by the Periphery Proposition Lemma \ref{lem:cyclesPropertyK} there can only be two such vertices, no others are present.
\end{remark}

The above results can be summarized by the following theorem.

\begin{theorem}
\label{thm:blackvertices}
Suppose there is a black two-valent vertex on the interior cycle $C_i$ in the interior graph $\Gamma_{i-1}$.  Then it resulted from excatly one of the following:
\begin{itemize}
	\item a black two-valent vertex on $C_{i-1}$,
	\item an ``accordion'' together with a black two-valent vertex on some other $C_{i+1}$, or
	\item a ``party hat'' together with the other black two-valent vertex on $C_i$.
\end{itemize}
\end{theorem}

%{\red FIGURE maybe?}

\begin{proof}
First use the Leaf Reduction Proposition \ref{prop:leafreduction} to remove the leaves.  Observe that this does not change the structure of the cycles in consideration here and in particular does not affect two-valent black vertices.

%Let $v_1$ be the black two-valent vertex on $C_i$ with neighbors $u_1$ and $u_2$ not on $C_i$.
%
%First suppose both $u_1$ and $u_2$ are on $C_{i-1}$, and consider the graph whose periphery is given by $u_1$, $v_1$, $u_2$, and the path along $C_{i-1}$ from $u_2$ to $u_1$ such that the three named vertices form a face inside this periphery and that $C_i$ is outside this graph.  If there are no interior cycles within this graph, then a black two-valent vertex is obtained on $C_{i-1}$ as in Proposition \ref{prop:2to2} and Figure \ref{fig:two-valent} {\red MAYBE PROVE THIS?}.  If there are interior cycles within this graph, 
%let $C_{i+1}$ be the {\red CLOSEST???} one. 
%
%Then the square face outside the interior graph $\Gamma_{i-1}$ containing $u_1$, $v_1$, and $u_2$ has one additional black vertex $v_2$.  This black vertex $v_2$ is either on $C_{i-1}$ or on some cycle $C_{i+1}$...

The Periphery Proposition \ref{prop:PropertyK} accounts for the two black two-valent vertices on $C_1$.  The Stacking Proposition \ref{prop:2to2} accounts for the two black two-valent vertices on each subsequent $C_i$ unless there are several interior cycles within some $C_{i-1}$.  These arise due to either disconnected components in $\Gamma_{i-1}\backslash C_{i-1}$ or breaking cutvertices.  The Accordion Proposition \ref{prop:accordion} accounts for two additional black two-valent vertices for each additional component, and the Party Hat Proposition \ref{prop:2b2break} accounts for two additional black two-valent vertices for each additional cycle after breaking cutvertices.

By the Periphery Proposition Lemma \ref{lem:cyclesPropertyK} there can be no more such black two-valent vertices.
\end{proof}

Now that we have the appropriate notions of accordions and party hats, the following reduction propositions can be introduced.

\begin{proposition}
\label{prop:accordionreduction}
\textbf{Accordion Reduction Proposition.}%Accordion Annulus Reduction Proposition.
Suppose that an accordion arises when $C_{i-1}$ is deleted from the interior graph $\Gamma_{i-1}$ as in the Accordion Proposition \ref{prop:accordion}.  Furthermore, suppose that the interior cycle $C_{i+1}$ is connected to $C_i$ by the accordion but is not connected to any other cycles.  Then the interior graph $\Gamma_{i-1}$ can be reduced to one without $C_{i+1}$ by a local move as depicted in Figure \ref{fig:accordionreduction}.
%Suppose the interior graph $\Gamma_{i-1}$ obtained from a knot diagram that has no nugatory crossings contains several interior cycles with some interior cycle $C_i$ connected to exactly one other cycle by an accordion.  Then this graph can be reduced to one without the cycle $C_i$ by a local move as depicted in Figure .
\end{proposition}

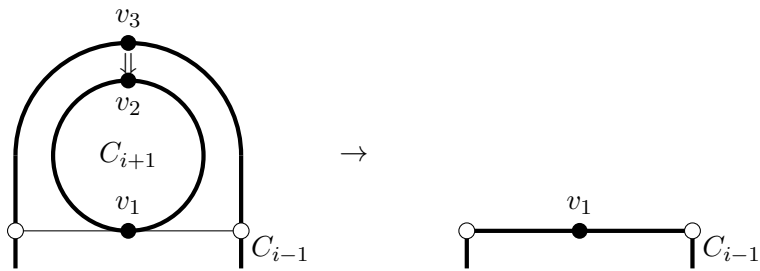
\begin{figure}[h]
\begin{center}
\begin{tikzpicture}

	\draw[ultra thick] (1.5,.5) circle (1);

\fill[color=black] (1.5,-.5) circle (3pt);% node[above] {$v_1$};
\fill[color=black] (1.5,1.5) circle (3pt);% node[below] {$v_2$};

\draw (1.5,-.16) node {$v_1$};
\draw (1.5,1.16) node {$v_2$};
\draw (1.5,2.33) node {$v_3$};

\draw[ultra thick] (3,.5) arc (0:180:1.5);
\fill[color=black] (1.5,2) circle (3pt) node[below] {$\Downarrow$}; %node[above] {$v_3$};

\draw[ultra thick] (0,.5) -- (0,-1);
\draw (0,-.5) -- (1.5,-.5) -- (3,-.5);
\draw[ultra thick] (3,-1) -- (3,.5);

	\fill[color=white] (0,-.5) circle (3pt);
	\draw (0,-.5) circle (3pt);
	\fill[color=white] (3,-.5) circle (3pt);
	\draw (3,-.5) circle (3pt);

\draw (3,-.75) node[right] {$C_{i-1}$};
\draw (1.5,.5) node {$C_{i+1}$};

\draw (4.5,.5) node {$\rightarrow$};

\draw[ultra thick] (6,-1) -- (6,-.5) -- (7.5,-.5) -- (9,-.5) -- (9,-1);
\draw (9,-.75) node[right] {$C_{i-1}$};

	\fill[color=white] (6,-.5) circle (3pt);
	\draw (6,-.5) circle (3pt);
	\fill[color=white] (9,-.5) circle (3pt);
	\draw (9,-.5) circle (3pt);

\fill[color=black] (7.5,-.5) circle (3pt);% node[above] {$v_1$};

\draw (7.5,-.16) node {$v_1$};

\end{tikzpicture}
	\caption{Applying the Accordion Reduction Proposition \ref{prop:accordionreduction}.}
	\label{fig:accordionreduction}
\end{center}
\end{figure}

%{\red Maybe choose better figures looking like the earlier ones?}

\begin{proof}
First use the Leaf Reduction Proposition \ref{prop:leafreduction} to remove the leaves.  Observe that this does not change the structure of the cycles in consideration here.

Recall that the accordion and the cycle $C_{i+1}$ share one black vertex $v_1$ that is two-valent on $C_{i+1}$.  Then since there are two such vertices on each cycle by the Periphery Proposition Lemma \ref{lem:cyclesPropertyK}, there is some other black two-valent vertex $v_2$ on $C_{i+1}$.  By Theorem \ref{thm:blackvertices} this must arise from a black two-valent vertex $v_3$ on $C_{i-1}$ since $C_{i+1}$ is not connected to any other cycles.

Delete all faces on and inside $C_{i+1}$ except the one that is part of the accordion on black vertex $v_1$.  Then this black vertex is now two-valent on $C_{i-1}$, replacing the former one $v_3$.  Furthermore, the cycle structure in $\Gamma_{i-1}$ outside of $C_{i+1}$ remains the same.
\end{proof}

\begin{remark}
%As in Remark \ref{rem:globloc}, o
One cannot employ the Accordion Reduction Proposition \ref{prop:accordionreduction} while considering the larger interior graph $\Gamma_{i-2}$.  Moreover, employing this reduction move ignores all of $\Gamma_{i+1}$.
\end{remark}

\begin{proposition}
\label{prop:partyhatreduction}
\textbf{Party Hat Reduction Proposition.}
Suppose that a party hat arises when $C_{i-1}$ is deleted from the interior graph $\Gamma_{i-1}$ as in the Party Hat Proposition \ref{prop:2b2break}.  Furthermore, suppose that the interior cycle $C_{i+1}$ is connected to $C_i$ by the party hat but is not connected to any other cycles.  Then the interior graph $\Gamma_{i-1}$ can be reduced to one without $C_{i+1}$ by a local move as depicted in Figure \ref{fig:partyhatreduction}.
\end{proposition}

\begin{figure}[h]
\begin{center}
\begin{tikzpicture}

\draw[ultra thick] (0,-4) -- (3,-4);
\draw (3.5,-4) node {$C_{i}$};

\draw[ultra thick] (0,-3) -- (.5,-3);
\draw (.5,-3) -- (.5,-2);
\draw[] (2.5,-2) -- (2.5,-3);
\draw[ultra thick] (2.5,-3) -- (3,-3);

\draw[ultra thick] (.5,-2) -- (.5,-3);
\draw[ultra thick] (2.5,-2) -- (2.5,-3);

\draw[ultra thick, loosely dotted] (2.5,-2) arc (0:180:1);
\draw (3.5,-3) node {$C_{i-1}$};

\draw (1.5,-2.5) node {$C_{i+1}$};

\draw (1.5,-1) -- (1.5,-1.5);

\draw (.5,-4) -- (.5,-3) -- (1,-3) -- (1.5,-4) -- (2,-3) -- (2.5,-3) -- (2.5,-4);

\draw[] (1,-3) -- (2,-3);
\draw[ultra thick, loosely dotted] (2,-2.5) arc (0:180:.5);
\draw[ultra thick, loosely dotted] (2,-3) -- (2,-2.5);
\draw[ultra thick, loosely dotted] (1,-3) -- (1,-2.5);
\draw[ultra thick] (1,-3) -- (2,-3);

\foreach \x/ \y in {.5/-3, 1.5/-4, 1.5/-3, 1.5/-2, 2.5/-3}
	{
	\fill[color=white] (\x,\y) circle (3pt);
	\draw (\x,\y) circle (3pt);
	}

\foreach \x/ \y in {.5/-4, 1/-3, 1.5/-1, 2/-3, 2.5/-4}
	{
	\fill[color=black] (\x,\y) circle (3pt);
	}

\draw (4.5,-3.5) node {$\rightarrow$};

\draw[ultra thick] (5.5,-4) -- (8.5,-4);
\draw (9,-4) node {$C_{i}$};

\draw[ultra thick] (5.5,-3) -- (8.5,-3);
\draw (9,-3) node {$C_{i-1}$};

\draw (6,-4) -- (6,-3) -- (6.5,-3) -- (7,-4) -- (7.5,-3) -- (8,-3) -- (8,-4);

\foreach \x/ \y in {6/-3, 7/-4, 7/-3, 8/-3}
	{
	\fill[color=white] (\x,\y) circle (3pt);
	\draw (\x,\y) circle (3pt);
	}

\foreach \x/ \y in {6/-4, 6.5/-3, 7.5/-3, 8/-4}
	{
	\fill[color=black] (\x,\y) circle (3pt);
	}

\draw (.91,-3.33) node {$v_1$};
\draw (2.09,-3.33) node {$v_2$};

\draw (6.41,-3.33) node {$v_1$};
\draw (7.59,-3.33) node {$v_2$};

\end{tikzpicture}
	\caption{Applying the Party Hat Reduction Proposition \ref{prop:partyhatreduction}.}
	\label{fig:partyhatreduction}
\end{center}
\end{figure}
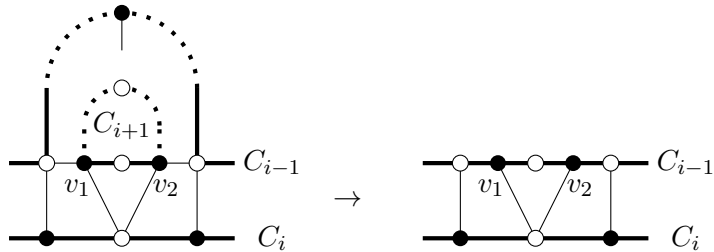

\begin{proof}
First use the Leaf Reduction Proposition \ref{prop:leafreduction} to remove the leaves.  Observe that this does not change the structure of the cycles in consideration here.

Recall that the party hat and the cycle $C_{i+1}$ share two black vertices $v_1$ and $v_2$ that are two-valent on $C_{i+1}$.  Then since there are only two such vertices on each cycle by the Periphery Proposition Lemma \ref{lem:cyclesPropertyK}, there are no others.  Thus by the Stacking Proposition \ref{prop:2to2} there cannot be any two-valent black vertices on $C_{i-1}$ in the vicinity of $C_{i+1}$.

Delete all faces on and inside $C_{i+1}$ except for those that are part of the cone of the party hat along with the two that are adjacent to it (one on $v_1$ and one on $v_2$).  Then these black vertices $v_1$ and $v_2$ are now on $C_{i-1}$ and three-valent there.  Furthermore, the cycle structure and black vertices in $\Gamma_{i-1}$ outside of $C_{i+1}$ remain the same.
\end{proof}

\begin{remark}
%As in Remark \ref{rem:globloc}, o
One cannot employ the Party Hat Reduction Proposition \ref{prop:partyhatreduction} while considering the larger interior graph $\Gamma_{i-2}$.  Moreover, employing this reduction move ignores all of $\Gamma_{i+1}$.
\end{remark}

These last two reduction moves will be used in the proof of Theorem \ref{conj:sum}.

%{\red MAYBE MAKE A SIMPLY CONNECTED REDUCTION MOVE!  Only connected to the rest by a single edge!}

%\begin{proposition}
%\label{prop:simpreduction}
%\textbf{Simply Connected Reduction Proposition.}
%Consider the interior graph $\Gamma_{i-1}$ obtained from a knot diagram that has no nugatory crossings.  Let $H$ be an induced subgraph with all its vertices on $C_{i-1}$ such that all but one of the edges of the periphery of $H$ are on the periphery $C_{i-1}$ of $\Gamma_{i-1}$.  Then this graph can be reduced to one without this region by a local move as depicted in Figure .
%Suppose the interior graph $\Gamma_{i-1}$ obtained from a knot diagram that has no nugatory crossings contains a simply connected induced subgraph with all its vertices on $C_{i-1}$.  Then this graph can be reduced to one without this region by a local move as depicted in Figure .
%\end{proposition}
%
%\begin{proof}
%Because all of the vertices of this simply connected region are on $C_{i-1}$, the boundary of 
%\end{proof}

\section{Using the cycles to study the graph of perfect matchings}
\label{sec:graphofpm}

These cycles $C_i$ emerge when the symmetric difference is taken of $\widehat{0}$ and $\widehat{1}$, the unique minimum and maximum elements in the graph $\mathcal{G}$ of perfect matchings of $\Gamma$ when seen as a lattice.

%These cycles $C_i$ emerge when $\widehat{0}$ and $\widehat{1}$ are superimposed. {\red want to say:  symmetric difference!}

%Furthermore, we can predict the ``clocked'' and ``counterclocked'' states, which are $\widehat{0}$ and $\widehat{1}$ in the lattice.

Decompose each cycle $C_i$ into two perfect matchings on the cycle subgraph:  the collection $\mu_i^0$ of edges that traverse from black to white in a clockwise direction and the collection $\mu_i^1$ of edges that traverse from black to white in a counterclockwise direction.  This orientation of course assumes the topological properties of the plane embedding of the original knot diagram.

%Decompose each cycle $C_i$ into two perfect matchings on the cycle subgraph:  the collection $C_i^0$ of edges that traverse from white to black in a clockwise direction and the collection $C_i^1$ of edges that traverse from white to black in a counterclockwise direction.
%{\red I THINK THIS IS WRONG above!}

\begin{definition}
A cycle is said to be \emph{$(\mu_1,\mu_2)$-alternating} if the edges appear alternately in the two matchings $\mu_1$ and $\mu_2$.
\end{definition}

\begin{theorem}
\label{thm:clocked}
Each $C_i$ is $(\widehat{0},\widehat{1})$-alternating.  Furthermore, the set of leaves is exactly the set of edges that appear in both of these states.
%Omitting the leaves, the ``clocked'' state $\widehat{0}$ is the union of $C_i^0$, and the ``counterclocked'' state $\widehat{1}$ is the union of $C_i^1$.  The leaves occur in both these states.
\end{theorem}

\begin{proof}
Consider first the union of the $\mu_i^0$ together with the leaves; to see this is the unique least element $\widehat{0}$ of Kauffman's clock lattice $\mathcal{L}$, it is enough to show that it cannot be counterclocked.  A counterclock move can only occur when two edges $e_i$ and $e_j$ (going clockwise from white vertex to black vertex on the boundary of the same square face $f$) belong to the pefect matching.

%{\red this might be backwards white-black above  -- no, it's ok above}

Recall that for an edge $e_i$ in $\widehat{0}$ to belong to a cycle $C_i$, it must go from black vertex to white vertex along the face within $\Gamma_i$.  Thus if the edge $e_i$ belongs to $C_i$, the face $f$ must lie outside of $\Gamma_i$; if this is the case for both $e_i$ and $e_j$, the cycles $C_i$ and $C_j$ are neighboring, but then the cycles $C_i$ and $C_j$ could have been extended through this face $f$ creating a larger cycle bounding $\Gamma_i\cup\Gamma_j\cup\{f\}$, a contradiction.  Then $e_j$ must be a leaf; if $e_i$ is not, the edge $e_j$ lies outside of $C_i$, but this could have been extended to a larger cycle bounding $\Gamma_i\cup\{f\}$ by the Inchworm Lemma \ref{lem:inchworm}.  Then $e_i$ is a leaf, as well, and a new cycle $C'_{ij}$ could have been created with interior $\Gamma'_{ij}=f$, a contradiction.

%{\red ``extended through $f$'' like inchworm lemma}

The proof that the union of the $\mu_i^1$ together with the leaves cannot be clocked is similar.
\end{proof}

See for example the clocked and counterclocked states of a knot diagram for $K11n157$ in Figure \ref{fig:thm:clocked}.  Notice that in particular when these are overlapped there are two concentric cycles and a single leaf in the center.

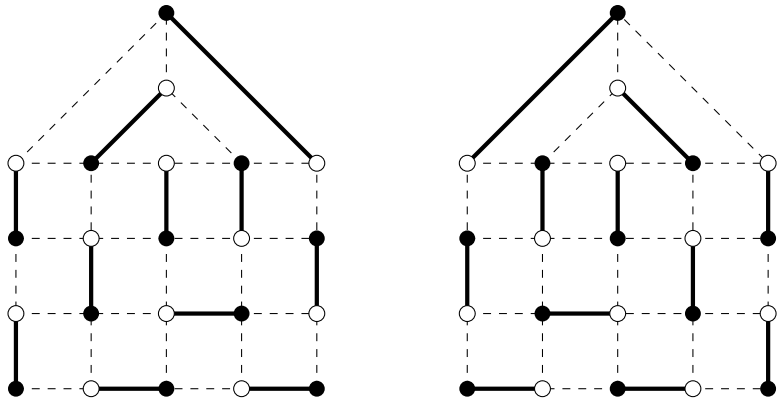
\begin{figure}[h]
\begin{center}
\begin{tikzpicture}

\foreach \x in {0,...,4}
	{
		\draw[dashed] (\x,0) -- (\x,3);
	}

\foreach \y in {0,...,3}
	{
		\draw[dashed] (0,\y) -- (4,\y);
	}

\draw[dashed] (0,3) -- (2,5) -- (4,3);
\draw[dashed] (2,4) -- (2,5);
\draw[dashed] (1,3) -- (2,4) -- (3,3);

\draw[ultra thick] (1,0) -- (2,0);
\draw[ultra thick] (3,0) -- (4,0);
\draw[ultra thick] (4,1) -- (4,2);
\draw[ultra thick] (4,3) -- (2,5);
\draw[ultra thick] (0,3) -- (0,2);
\draw[ultra thick] (0,1) -- (0,0);

\draw[ultra thick] (2,1) -- (3,1);
\draw[ultra thick] (3,2) -- (3,3);
\draw[ultra thick] (1,3) -- (2,4);
\draw[ultra thick] (1,1) -- (1,2);

\draw[ultra thick] (2,2) -- (2,3);

%\draw (0,0) node {$\widehat{0}$};
%\draw (3,0) node {$\widehat{0}\triangle\widehat{1}$};
%\draw (6,0) node {$\widehat{1}$};

\foreach \x in {6,...,10}
	{
		\draw[dashed] (\x,0) -- (\x,3);
	}

\foreach \y in {0,...,3}
	{
		\draw[dashed] (6,\y) -- (10,\y);
	}

\draw[dashed] (6,3) -- (8,5) -- (10,3);
\draw[dashed] (8,4) -- (8,5);
\draw[dashed] (7,3) -- (8,4) -- (9,3);

\draw[ultra thick] (6,0) -- (7,0);
\draw[ultra thick] (8,0) -- (9,0);
\draw[ultra thick] (10,0) -- (10,1);
\draw[ultra thick] (10,2) -- (10,3);
\draw[ultra thick] (8,5) -- (6,3);
\draw[ultra thick] (6,1) -- (6,2);

\draw[ultra thick] (7,1) -- (8,1);
\draw[ultra thick] (9,1) -- (9,2);
\draw[ultra thick] (9,3) -- (8,4);
\draw[ultra thick] (7,3) -- (7,2);

\draw[ultra thick] (8,2) -- (8,3);

\foreach \x/ \y in {0/0, 2/0, 4/0, 1/1, 3/1, 0/2, 2/2, 4/2, 1/3, 3/3, 2/5}
    {
	\fill[color=black] (\x,\y) circle (3pt);
		}

\foreach \x/ \y in {1/0, 3/0, 0/1, 2/1, 4/1, 1/2, 3/2, 0/3, 2/3, 4/3, 2/4}
    {
	\fill[color=white] (\x,\y) circle (3pt);
	\draw (\x,\y) circle (3pt);
    }

\foreach \x/ \y in {6/0, 8/0, 10/0, 7/1, 9/1, 6/2, 8/2, 10/2, 7/3, 9/3, 8/5}
    {
	\fill[color=black] (\x,\y) circle (3pt);
		}

\foreach \x/ \y in {7/0, 9/0, 6/1, 8/1, 10/1, 7/2, 9/2, 6/3, 8/3, 10/3, 8/4}
    {
	\fill[color=white] (\x,\y) circle (3pt);
	\draw (\x,\y) circle (3pt);
    }

\end{tikzpicture}
	\caption{The clocked and counterclocked states of a diagram for $K11n157$.}
	\label{fig:thm:clocked}
\end{center}
\end{figure}

Now we arrive at the main theorem.

\begin{theorem}%[Main Theorem]
\label{conj:sum}
Consider the balanced overlaid Tait graph $\Gamma$ obtained from a knot diagram, and let $s(C_i)$ be the number of square faces within the interior graph $\Gamma_i$.  Then
\begin{equation}
\label{eq:sum}
\sum_{i}s(C_i)=h
\end{equation}
gives the height of the clock lattice.
\end{theorem}

In particular, $s(C_1)$ is equal to the original number of square faces of the graph $\Gamma$.

\begin{proof}
%Observe that
Since $\mathcal{G}$ is connected, there is always at least one possible flip move to make, so $s(C_i)\neq 0$.

Proceed by induction on $k$, the number of cycles.  The base case is handled by Lemma \ref{lem:basecase}, which also shows that the Simply Connected Region Reduction Proposition \ref{prop:unstacking} does not affect Equation \ref{eq:sum} in Lemma \ref{lem:mainunstackingreduction}.  The induction hypothesis on an annulus is then shown in Lemma \ref{lem:indhyp}, barring the Accordion Reduction Proposition \ref{prop:accordionreduction} handled in Lemma \ref{lem:mainaccordionreduction}, the Party Hat Reduction Proposition \ref{prop:partyhatreduction} handled in Lemma \ref{lem:mainpartyhatreduction}, and the Leaf Reduction Proposition \ref{prop:leafreduction} handled in Lemma \ref{lem:mainleafreduction}.

%The idea is to start from $\widehat{0}$, find a single face to flip, then use successive flips to change $C_1^0$ to $C_1^1$ by flipping all internal squares.  Then repeat the process for each $C_i$ (for $i<k$ when $C_k$ is a tree).

\begin{lemma}
\label{lem:basecase}
\textbf{Base Case: One Cycle.}  
A knot diagram with exactly one connected cycle $C$ has clock lattice height of $s(C)$.
\end{lemma}

%{\red 
%Consider changing $o$ to $outer$ and $i$ to $inner$ in proof and figure below!
%}

\begin{proof}
Induct on the number of squares $s(C)$ within the only cycle $C$; as a base case, a single square has height one.

%{\red CHANGE BELOW TO MAKE IT ADDING A SQUARE NOT DELETING IT! -- Maybe I can't do this, though.}

Suppose there are $s(C)=m$ squares within the cycle and that the induction hypothesis holds for all cycles containing fewer than $m$ squares.  Choose some square face $f_1$ sharing at least one edge with $C$; this produces a new cycle $C'=C\triangle f_1$ within $C$ by the symmetric difference.

%this way the new cycle $C'=C\triangle s_0$ produced by the symmetric difference is indeed a single connected component, as well.  Then the height of $C'$ is $m-1$ and it is sufficient to show that including $s_0$ adds exactly one to the height.

%{\red Um, maybe need to say that if $C'$ is not connected (below), then we have non-simply connected interior graph...}

The square face $f_1$ cannot share all its four edges with $C$.  It must share consecutive edges; otherwise $C'$ would not be connected and there would be more than one cycle.

Let $P_{outer}$ (respectively $P_{inner}$) be the path formed by those consecutive edges on $f_1$ shared with $C$ (respectively $C'$), and partition this path into $\mu_{outer}^0$ and $\mu_{outer}^1$ (respectively $\mu_{inner}^0$ and $\mu_{inner}^1$), those edges that traverse from white to black in a clockwise direction around $C$ (respectively $C'$) and, respectively, those edges that traverse from white to black in a counterclockwise direction.  Observe that $P_{outer}$ is ``outside'' of $f_1$ and $P_{inner}$ is on the ``inside'', as in Figure \ref{fig:basecase} (A), where black and white vertices are undistinguished throughout.  By the induction hypothesis, there must be some sequence $F$ of $m-1$ flip moves that transfer $\mu_{C'}^0$ to $\mu_{C'}^1$ (and in particular $\mu_{inner}^0$ to $\mu_{inner}^1$), as $P_{inner}$ belongs to $C'$.

\begin{figure}
\begin{center}
\begin{tikzpicture}

\draw (0,2.5) node {$(A)$};
\draw (4.5,2.5) -- (5,2.5);
\draw (7.5,2.5) -- (7,2.5);
\draw (6,2.5) ellipse (1 and .5);
\fill[color=black] (5,2.5) circle (3pt);
\fill[color=black] (7,2.5) circle (3pt);
\draw (6,2.5) node {$f_1$};
\draw (6,3.25) node {$P_{outer}$};
\draw (6,1.75) node {$P_{inner}$};
\draw (6.85,3) node {$C$};
\draw (6.85,2) node {$C'$};

\draw (0,.5) node {$(B)$};

\draw[dashed] (.5,0) -- (3.5,0);
\draw[dashed] (1,0) -- (1,1) -- (3,1) -- (3,0);
\draw (.5,0) -- (1,0);
\draw (1,1) -- (3,1);
\draw (3,0) -- (3.5,0);
\fill[color=black] (1,0) circle (3pt);
\fill[color=black] (3,0) circle (3pt);
\fill[color=black] (1,1) circle (3pt);
\fill[color=black] (3,1) circle (3pt);

\draw (2,-.2) node {$P_{inner}$};
\draw (2,.5) node {$f_1$};
\draw[<->] (3.85,.25) -- (4.15,.25);
\draw (4,.55) node {$F$};

\draw[dashed] (4.5,0) -- (7.5,0);
\draw[dashed] (5,0) -- (5,1) -- (7,1) -- (7,0);
\draw (5,1) -- (7,1);
\draw (5,0) -- (7,0);
\fill[color=black] (5,0) circle (3pt);
\fill[color=black] (7,0) circle (3pt);
\fill[color=black] (5,1) circle (3pt);
\fill[color=black] (7,1) circle (3pt);

\draw[<->] (7.85,.25) -- (8.15,.25);
\draw (8,.5) node {$f_1$};

\draw[dashed] (8.5,0) -- (11.5,0);
\draw[dashed] (9,0) -- (9,1) -- (11,1) -- (11,0);
\draw (9,0) -- (9,1);
\draw (11,0) -- (11,1);
\fill[color=black] (9,0) circle (3pt);
\fill[color=black] (11,0) circle (3pt);
\fill[color=black] (9,1) circle (3pt);
\fill[color=black] (11,1) circle (3pt);

\draw (0,-1.5) node {$(C)$};

\draw[dashed] (.5,-2) -- (3.5,-2);
	\draw[dashed] (1,-2) .. controls (1.66,-1) and (2.33,-1) .. (3,-2);
\draw (.5,-2) -- (1,-2);
\draw (1.66,-2) -- (2.33,-2);
\draw (3,-2) -- (3.5,-2);
\fill[color=black] (1,-2) circle (3pt);
\fill[color=black] (1.66,-2) circle (3pt);
\fill[color=black] (2.33,-2) circle (3pt);
\fill[color=black] (3,-2) circle (3pt);

\draw (2,-1) node {$P_{outer}$};
\draw (2,-1.5) node {$f_1$};
\draw[<->] (3.85,-1.75) -- (4.15,-1.75);
\draw (4,-1.45) node {$F$};

\draw[dashed] (4.5,-2) -- (7.5,-2);
	\draw[dashed] (5,-2) .. controls (5.66,-1) and (6.33,-1) .. (7,-2);
\draw (5,-2) -- (5.66,-2);
\draw (6.33,-2) -- (7,-2);
\fill[color=black] (5,-2) circle (3pt);
\fill[color=black] (5.66,-2) circle (3pt);
\fill[color=black] (6.33,-2) circle (3pt);
\fill[color=black] (7,-2) circle (3pt);

\draw[<->] (7.85,-1.75) -- (8.15,-1.75);
\draw (8,-1.5) node {$f_1$};

\draw[dashed] (8.5,-2) -- (11.5,-2);
	\draw (9,-2) .. controls (9.66,-1) and (10.33,-1) .. (11,-2);
\draw (9.66,-2) -- (10.33,-2);
\fill[color=black] (9,-2) circle (3pt);
\fill[color=black] (9.66,-2) circle (3pt);
\fill[color=black] (10.33,-2) circle (3pt);
\fill[color=black] (11,-2) circle (3pt);

\draw (1.66,-2.33) node {$v_1$};
\draw (2.33,-2.33) node {$v_2$};

\draw (0,-3.5) node {$(D)$};

\draw (.5,-4) -- (3.5,-4);
\draw (1,-4) -- (1.66,-3) -- (2.33,-4);
\fill[color=black] (1,-4) circle (3pt);
\fill[color=black] (1.66,-4) circle (3pt);
\fill[color=black] (1.66,-3) circle (3pt);
\fill[color=black] (2.33,-4) circle (3pt);
\fill[color=black] (3,-4) circle (3pt);

\draw (1.66,-3.5) node {$f_1$};
\draw[<->] (3.85,-3.75) -- (4.15,-3.75);
\draw (4,-3.45) node {$f_2$};

\draw (4.5,-4) -- (7.5,-4);
\draw (5,-4) -- (5.66,-3) -- (6.33,-4);
\draw (5.66,-3) -- (6.33,-3) -- (7,-4);
\fill[color=black] (5,-4) circle (3pt);
\fill[color=black] (5.66,-4) circle (3pt);
\fill[color=black] (5.66,-3) circle (3pt);
\fill[color=black] (6.33,-4) circle (3pt);
\fill[color=black] (6.33,-3) circle (3pt);
\fill[color=black] (7,-4) circle (3pt);

\draw (5.66,-3.5) node {$f_1$};
\draw (6.33,-3.5) node {$f_2$};
\draw[<->] (7.85,-3.75) -- (8.15,-3.75);
\draw (8,-3.45) node {$f_1$};

\draw (8.5,-4) -- (11.5,-4);
\draw (10.33,-4) -- (9.66,-3) -- (10.33,-3) -- (11,-4);
\fill[color=black] (9,-4) circle (3pt);
\fill[color=black] (9.66,-4) circle (3pt);
\fill[color=black] (9.66,-3) circle (3pt);
\fill[color=black] (10.33,-4) circle (3pt);
\fill[color=black] (10.33,-3) circle (3pt);
\fill[color=black] (11,-4) circle (3pt);
\draw (10.33,-3.5) node {$f_2$};

\end{tikzpicture}
	\caption{%Moving between the clocked and counterclocked states of the only cycle $C$.}
		Lemma \ref{lem:basecase} provides the base case for Theorem \ref{conj:sum}, adding new square face $f_1$.  The vertex sets are undistinguished.}
	\label{fig:basecase}
\end{center}
\end{figure}
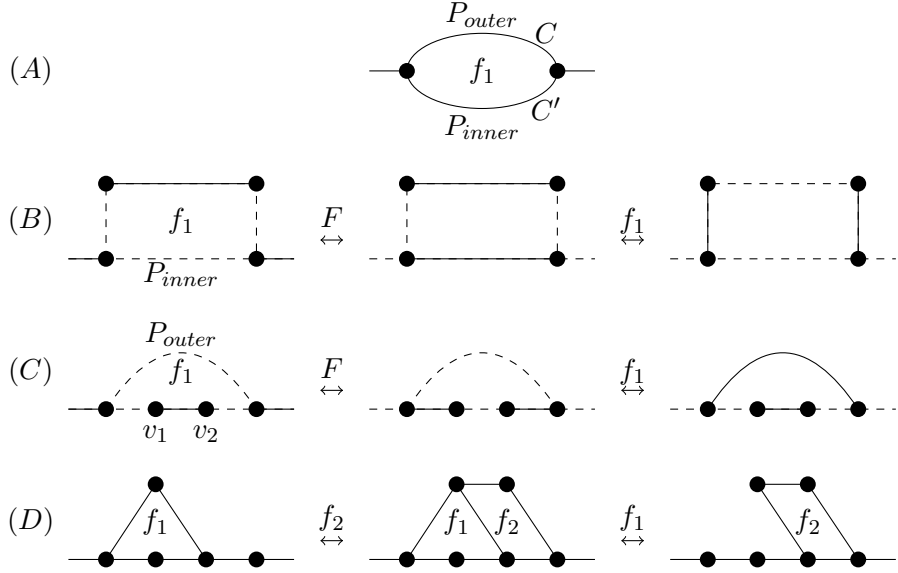

First suppose that $|P_{inner}|=1$; then either $f_1$ followed by $F$ or $F$ followed by $f_1$ transfers $\mu^0_{C}$ to $\mu^1_{C}$ in a total of $m$ moves, as in Figure \ref{fig:basecase} (B).  The other cases involve the inclusion of new leaves.

Supposing that $|P_{inner}|=3$, there are two additional vertices $v_1$ and $v_2$ in $P_{inner}$ that are inside of $C$ but not contained in $C$.  If there are no other leaves inside of $C$, then the clocked state of $C$ is $\mu^0_{C}$ together with the leaf $\ell$ between $v_1$ and $v_2$.  Then as before either $f_1$ followed by $F$ or $F$ followed by $f_1$ transfers $\mu^0_C \amalg \ell$ to $\mu^1_C \amalg \ell$ in $m$ moves, as in Figure \ref{fig:basecase} (C).  Otherwise it may be possible for this leaf to share a face with two other vertices in a leaf; however this would form a new cycle $C''$ within $C$, a contradiction, and thus there is just a single additional leaf.

%WRONG -- for proof of $\sum s(C_i)$
%It cannot be that $|P_i|=3$; otherwise there is a black vertex in $P_i$ but not $C$ that is at most three-valent in the graph $\Gamma'$ whose outer face is $C'$ but then cannot be more than three-valent in the graph $\Gamma$ whose outer face is $C$, contradicting Proposition \ref{prop:concentriccircles}.
%Argue:
% - need two such new squares, since no other way to add an odd number of vertices.
% - maybe can show adding two such squares = adding one of type III.

%If $|P_i|=2$, there is now a single additional vertex $v_1$ in $P_i$ that is inside of $C$ but not contained in $C$.  However, $C'=C\triangle s_0$
%For $|P_i|=2$, there is a single vertex $v_1$ 
%It cannot be that $|P_i|=2$; {\red not equal sized vertex sets, so not coming from a knot diagram}

Removing a square $f_1$ with $|P_{inner}|=2$ yields a bipartite graph with unequally-sized vertex sets; thus two such squares $f_1$ and $f_2$ need to be removed at a time since no other case allows for unequally-sized vertex sets.  These squares may not share more than two of the same edges since the ``inside'' paths $P_{inner}$ cannot overlap, and the case where both $P_{inner}$ are the same two edges violates Theorem \ref{prop:concentriccircles}.  These squares $f_1$ and $f_2$ cannot share two of the same edges, as this would not yield equally-sized vertex sets.  When $f_1$ and $f_2$ are disjoint this also violates Theorem \ref{prop:concentriccircles}.

%\begin{sublemma}can't have $|P_i|=2$ except for extending a two-vertex\end{sublemma}
%\begin{sublemma}2+2 is 1+3\end{sublemma}

Therefore $f_1$ and $f_2$ must share a single edge as depicted in Figure \ref{fig:basecase} (D); deleting first $f_2$ (as on the left side) and then $f_1$ yields the same result as deleting first $f_1$ (as on the right side), which has $|P_{inner}|=3$, and then $f_2$, which has $|P_{inner}|=1$; thus here the $|P_i|=2$ squares are unneccessary.
\end{proof}

This can also be used to handle any simply connected region.

\begin{lemma}
\label{lem:mainunstackingreduction}
\textbf{Invariance of equality under the Simply Connected Region Reduction Proposition.}
The Simply Connected Region Reduction Proposition \ref{prop:unstacking} does not affect the equality of Theorem \ref{conj:sum}.
\end{lemma}

\begin{proof}
Following the proof of Lemma \ref{lem:basecase} above for the $|P_{inner}|=1$ case, extend the single square $f_1$ to any set of simply connected squares.  Extend the path $P_{outer}$ from length three in the case above to any odd length (since the graph is bipartite) in the natural way, and the result holds by applying Lemma \ref{lem:basecase}.
\end{proof}

We proceed to the induction step:  a single annulus between two cycles $C_{i-1}$ and $C_i$ with no leaves.

\begin{lemma}
\label{lem:indhyp}
\textbf{Induction Hypothesis:  Flipping an Annulus.}  
Suppose the interior graph $\Gamma_{i-1}$ has exactly one cycle $C_i$ at the next level inside of it.  Flipping all the square faces in $\Gamma_{i-1}\backslash\Gamma_i$ exactly once takes the local perfect matchings of $\mu_{i-1}^0$ and $\mu_i^1$ to those of $\mu_{i-1}^1$ and $\mu_i^0$.
\end{lemma}

\begin{proof}
Assume for now that there are no leaves in the annulus.  These will be taken care of by another lemma below.
%First use the Leaf Reduction Proposition \ref{prop:leafreduction} to remove the leaves.  Observe that this does not change the structure of the cycles in consideration here or the equality of Theorem \ref{conj:sum} by Lemma \ref{lem:mainleafreduction}.

This annulus has outside face $C_{i-1}$ and inside face $C_i$.  Because there are no leaves, all vertices are on $C_{i-1}$ and $C_i$, so this gives five types of square faces as in Figure \ref{fig:annulussquaretypes}.  First observe that any edge between two vertices of $C_{i-1}$ (respectively two vertices of $C_i$) encloses a simply connected region of squares with $C_{i-1}$ (respectively $C_i$), and so by Lemma \ref{lem:mainunstackingreduction} this does not affect the equality of Theorem \ref{conj:sum}.
% \ref{lem:basecase} we may consider this edge to actually belong to $C_{i-1}$ (respectively $C_i$).  This is the first example of \emph{pruning} a simply connected region.  
This leaves only cases III, IV, and V, where edges must be on $C_{i-1}$ or $C_i$, or must traverse them.

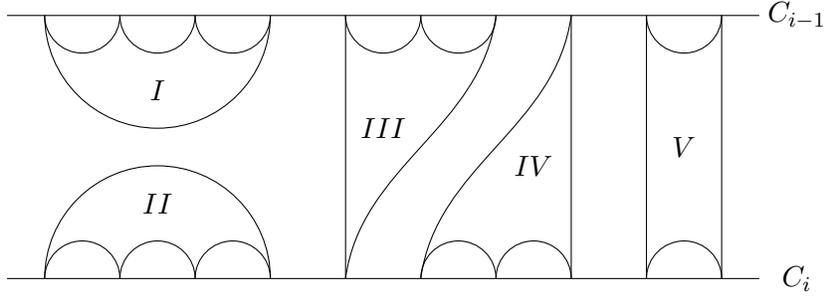
\begin{figure}[h]
\begin{center}
\begin{tikzpicture}

\draw (.5,0) -- (10.5,0);
	\draw (11,0) node {$C_{i}$};
\draw (.5,3.5) -- (10.5,3.5);
	\draw (11,3.5) node {$C_{i-1}$};
\draw (2,0) arc (0:180:.5);
\draw (3,0) arc (0:180:.5);
\draw (4,0) arc (0:180:.5);
\draw (4,0) arc (0:180:1.5);
	\draw (2.5,1) node {$II$};
\draw (1,3.5) arc (180:360:.5);
\draw (2,3.5) arc (180:360:.5);
\draw (3,3.5) arc (180:360:.5);
\draw (1,3.5) arc (180:360:1.5);
	\draw (2.5,2.5) node {$I$};
\draw (5,0) -- (5,3.5);
\draw (5,0) .. controls (5.2,1.5) and (6.8,2) .. (7,3.5);
\draw (5,3.5) arc (180:360:.5);
\draw (6,3.5) arc (180:360:.5);
	\draw (5.5,2) node {$III$};
\draw (6,0) .. controls (6.2,1.5) and (7.8,2) .. (8,3.5);
\draw (7,0) arc (0:180:.5);
\draw (8,0) arc (0:180:.5);
\draw (8,0) -- (8,3.5);
	\draw (7.5,1.5) node {$IV$};
\draw (9,0) -- (9,3.5);
\draw (10,0) arc (0:180:.5);
\draw (9,3.5) arc (180:360:.5);
\draw (10,0) -- (10,3.5);
	\draw (9.5,1.75) node {$V$};

%\fill[color=black] (5,2.5) circle (3pt);
%	\draw (9,-2) .. controls (9.66,-1) and (10.33,-1) .. (11,-2);

\end{tikzpicture}
	\caption{The five types of squares on an annulus with no leaves.}
	\label{fig:annulussquaretypes}
\end{center}
\end{figure}

Partition these latter edges into two sets:  those with a black vertex on $C_i$ and those with a black vertex on $C_{i-1}$.  Consider the oriented dual edges to these edges, directed to the right when leaving the black vertex.  Then every vertex in this ``dual-in-the-annulus'' graph has valence two, allowing for sources and sinks.

Begin by performing clock moves on all sources, following these oriented edges in the dual graph outwards, and performing a clock move on a sink when it is reached from both sides simultaneously.  Thus each square is counted exactly once, and so this changes $\mu_{i-1}^0$ to $\mu_{i-1}^1$ and $\mu_i^1$ to $\mu_i^0$.
\end{proof}

However it may be the case that there are several interior cycles at the same level; these occur with accordions and party hats.  In each of these configurations, we consider a single additional cycle, turn it into an annulus, and then apply Lemma \ref{lem:indhyp}.

\begin{lemma}
\label{lem:mainaccordionreduction}
\textbf{Invariance of equality under the Accordion Reduction Proposition.}
The Accordion Reduction Proposition \ref{prop:accordionreduction} does not affect the equality of Theorem \ref{conj:sum}.
\end{lemma}

\begin{proof}
Consider Figure \ref{fig:accordionreduction}, and let $H$ be the induced subgraph containing $\Gamma_{i+1}$ and the path on the cycle $C_{i-1}$ that is deleted.  This is \emph{almost} an annulus; to turn this into an annulus $H'$ one need only add a single square face $f$ at the two-valent vertex on the cycle $C_{i+1}$ that is on the periphery of $H$ as in Figure \ref{fig:accordionreductionlemma}.

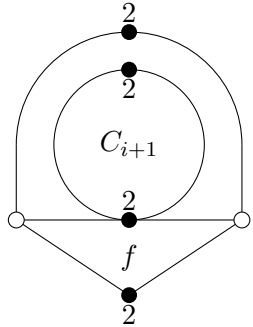
\begin{figure}[h]
\begin{center}
\begin{tikzpicture}

	\draw (1.5,.5) circle (1);

\fill[color=black] (1.5,-.5) circle (3pt) node[above] {$2$};
\fill[color=black] (1.5,1.5) circle (3pt) node[below] {$2$};

\draw (3,.5) arc (0:180:1.5);
\fill[color=black] (1.5,2) circle (3pt) node[above] {$2$};% node[below] {$\Downarrow$};

\draw (0,.5) -- (0,-.5) -- (1.5,-.5) -- (3,-.5) -- (3,.5);

%\draw (3,-.75) node[right] {$C_{i-1}$};
\draw (1.5,.5) node {$C_{i+1}$};

\draw (0,-.5) -- (1.5,-1.5) -- (3,-.5);
\fill[color=black] (1.5,-1.5) circle (3pt) node[below] {$2$};

\draw (1.5,-1) node {$f$};

	\fill[color=white] (0,-.5) circle (3pt);
	\draw (0,-.5) circle (3pt);
	\fill[color=white] (3,-.5) circle (3pt);
	\draw (3,-.5) circle (3pt);
	
\end{tikzpicture}
	\caption{Adding a square face $f$ to the deleted subgraph $H$ of the Accordion Reduction Proposition \ref{prop:accordionreduction} yields an annulus $H'$.}
	\label{fig:accordionreductionlemma}
\end{center}
\end{figure}

Observe that $H'$ has two black two-valent vertices by construction and since the remaining black vertices on the periphery were also on the periphery of $\Gamma_{i-1}$, this annulus $H'$ satisfies the Periphery Proposition Lemma \ref{lem:cyclesPropertyK}.

Thus we can apply Lemma \ref{lem:indhyp} to show that each of the square faces in this annulus contribute exactly once to the height.  It is easy to see that deleting the additional face $f$ does not affect this equality.
\end{proof}

\begin{lemma}
\label{lem:mainpartyhatreduction}
\textbf{Invariance of equality under the Party Hat Reduction Proposition.}
The Party Hat Reduction Proposition \ref{prop:partyhatreduction} does not affect the equality of Theorem \ref{conj:sum}.
\end{lemma}

\begin{proof}
As in the proof of Lemma \ref{lem:mainaccordionreduction}, consider Figure \ref{fig:partyhatreduction}, and let $H$ be the induced subgraph containing $\Gamma_{i+1}$ and the path on the cycle $C_{i-1}$ that is deleted.  This again is \emph{almost} an annulus; to turn this into an annulus $H'$ one need only add some square faces $f_1,\ldots,f_n$ at the two black two-valent vertices on the cycle $C_{i+1}$ that is on the periphery of $H$ as in Figure \ref{fig:partyhatreductionlemma}.

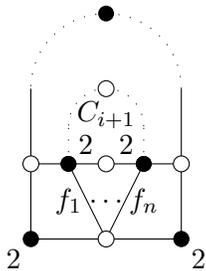
\begin{figure}[h]
\begin{center}
\begin{tikzpicture}

\draw[] (.5,-4) -- (2.5,-4);
%\draw (3.5,-4) node {$C_{i+1}$};

\draw[] (.5,-3) -- (.5,-2);
\draw[] (2.5,-2) -- (2.5,-3);
\draw[loosely dotted] (2.5,-2) arc (0:180:1);
%\draw (3.5,-3) node {$C_{i-1}$};

\draw (1.5,-2.35) node {$C_{i+1}$};

\draw (.5,-4) -- (.5,-3) -- (1,-3) -- (1.5,-4) -- (2,-3) -- (2.5,-3) -- (2.5,-4);

\draw[] (1,-3) -- (2,-3);
\draw[loosely dotted] (2,-2.5) arc (0:180:.5);
\draw[loosely dotted] (2,-3) -- (2,-2.5);
\draw[loosely dotted] (1,-3) -- (1,-2.5);

\foreach \x/ \y in {.5/-3, 1.5/-4, 1.5/-3, 1.5/-2, 2.5/-3}
	{
	\fill[color=white] (\x,\y) circle (3pt);
	\draw (\x,\y) circle (3pt);
	}

\foreach \x/ \y in {.5/-4, 1/-3, 1.5/-1, 2/-3, 2.5/-4}
	{
	\fill[color=black] (\x,\y) circle (3pt);
	}

\draw (1,-3.5) node {$f_1$};
\draw (1.5,-3.5) node {$\ldots$};
\draw (2,-3.5) node {$f_n$};
\fill[color=black] (.5,-4) circle (3pt) node[below left] {$2$};
\fill[color=black] (1,-3) circle (3pt) node[above right] {$2$};
\fill[color=black] (2,-3) circle (3pt) node[above left] {$2$};
\fill[color=black] (2.5,-4) circle (3pt) node[below right] {$2$};

\end{tikzpicture}
	\caption{Adding square faces $f_1,\ldots,f_n$ to the deleted subgraph $H$ of the Party Hat Reduction Proposition \ref{prop:partyhatreduction} yields an annulus $H'$.}
	\label{fig:partyhatreductionlemma}
\end{center}
\end{figure}

Observe again that $H'$ has two black two-valent vertices by construction and since the remaining black vertices on the periphery were also on $C_{i-1}$ or $C_{i+1}$, this annulus $H'$ satisfies the Periphery Proposition Lemma \ref{lem:cyclesPropertyK}.

Thus as above we can apply Lemma \ref{lem:indhyp} to show that each of the square faces in this annulus contribute exactly once to the height.  It is easy to see that deleting the additional faces $f_1,\ldots,f_n$ does not affect this equality.
\end{proof}

Lastly we show that we can remove leaves to arrive at any of the cases above.

%Before we proceed to the induction step, we first remove leaves.

\begin{lemma}
\label{lem:mainleafreduction}
\textbf{Invariance of equality under the Leaf Reduction Proposition.}
Equation \ref{eq:sum} is preserved by the Leaf Reduction Proposition \ref{prop:leafreduction}.
%The Leaf Reduction Proposition \ref{prop:leafreduction} does not affect the equality of Theorem \ref{conj:sum}.
\end{lemma}

\begin{proof}
%...simple direct calculation, done as generally as possible...
Suppose there is a leaf in $\Gamma_{i-1}$ as in Figure \ref{fig:leafreductionlemma}, with square faces labelled $f_1$ to $f_n$ along the upper part of the figure leaving the two faces $f_{n+1}$ and $f_{n+2}$ on the lower part.  Label the boundary $\partial H$ of this induced subgraph $H$ in a counterclockwise manner $e_1\cup P_{upper}\cup e_n\cup P_{lower}$ starting at the right.  We suppose that faces outisde of the induced subgraph $H$ act as we need so that the edges in the figure can be considered in perfect matchings.

%Let $\mu_A$ be $\mu_0|_{\partial H}$ and let $\mu_B$ be $\mu_1|_{\partial H}$.  We first show that each square in the induced subgraph $H$ must be flipped exactly once to transfer the edges from $\mu_A$ to $\mu_B$ where $\widehat{0}\leq \mu_A < \mu_B \leq \widehat{1}$.
%{\red How about this notation?}

We first show that each square face in the induced subgraph $H$ must be flipped exactly once to transfer the edges from $\mu^0_{\partial H}$ to $\mu^1_{\partial H}$ where $\widehat{0}\leq \mu^0_{\partial H} < \mu^1_{\partial H} \leq \widehat{1}$ with partial order given by Kauffman's lattice.

Observe that the leaf belongs to the clocked state $\widehat{0}$ (and so no other edges on these vertices can belong to $\widehat{0}$); thus the only possible flip move here is $f_1$, since we may assume that $e_1$ is in $\mu^0_{\partial H}$.  %depending on whether previous flip moves have put the rightmost edge of the figure into the perfect matching.
Then the faces $f_2,\ldots, f_{n-1}$ can be flipped successively along the upper part and the faces $f_{n+1}$ and $f_{n+2}$ can be flipped successively along the lower part.  Finally $f_n$ can be flipped, and the matching on the boundary is indeed $\mu^1_{\partial H}$.

Next see that after the reduction move, each square in the new induced subgraph $H'$ must be flipped exactly once, $f_1,\ldots,f_n$ successively, to transfer the edges from $\mu^0_{\partial H'}$ to $\mu^1_{\partial H'}$.

Lastly observe that all of these square faces belong to the same interior graph $\Gamma_{i-1}$ but not to any smaller nested interior graphs $\Gamma_i,\Gamma_{i+1},\ldots$ so they are counted exactly once per cycle.

%{\red Did I really prove this?}

If there are several leaves, one of these reduction moves may be performed at a time.
\end{proof}

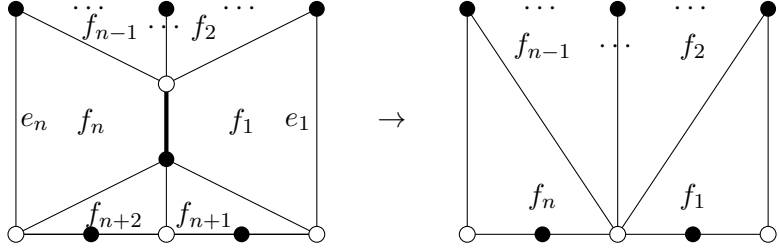
\begin{figure}[h]
\begin{center}
\begin{tikzpicture}

\draw[ultra thick] (2,1) -- (2,2);
\draw (0,0) -- (4,0) -- (4,3) -- (2,2) -- (0,3) -- (0,0) -- (2,1) -- (4,0) -- cycle;
\draw (2,0) -- (2,3);
\draw (5,1.5) node {$\rightarrow$};
\draw (6,0) -- (10,0) -- (10,3) -- (8,0) -- (6,3) -- cycle;
\draw (8,0) -- (8,3);
\foreach \x/ \y in {0/0, 2/0, 4/0, 2/2, 6/0, 8/0, 10/0}
	{			\fill[color=white] (\x,\y) circle (3pt);
				\draw (\x,\y) circle (3pt);											}
\foreach \x/ \y in {1/0, 3/0, 2/1, 0/3, 2/3, 4/3, 7/0, 9/0, 6/3, 8/3, 10/3}
	{			\fill[color=black] (\x,\y) circle (3pt);				}
\draw (1,3) node {$\cdots$};
\draw (3,3) node {$\cdots$};
\draw (7,3) node {$\cdots$};
\draw (9,3) node {$\cdots$};

\draw (3.75,1.5) node {$e_1$};
\draw (3,1.5) node {$f_1$};
\draw (2.5,2.75) node {$f_2$};
\draw (2,2.75) node {$\cdots$};
\draw (1.25,2.75) node {$f_{n-1}$};
\draw (1,1.5) node {$f_n$};
\draw (.25,1.5) node {$e_n$};
\draw (2.5,.25) node {$f_{n+1}$};
\draw (1.3,.25) node {$f_{n+2}$};

\draw (9,.5) node {$f_1$};
\draw (9,2.5) node {$f_2$};
\draw (8,2.5) node {$\cdots$};
\draw (7,2.5) node {$f_{n-1}$};
\draw (7,.5) node {$f_n$};

\end{tikzpicture}
	\caption{The Leaf Reduction Proposition \ref{prop:leafreduction} already has a simply connected region $H$.}
	\label{fig:leafreductionlemma}
\end{center}
\end{figure}

%The induction step here will be the annulus 
%{\purple
%We now introduce the first of several ``reduction moves'' that will decrease the number of square faces inside an interior graph $\Gamma_{i-1}$ when there is at least one interior cycle $C_i$ within the graph.  These moves will be used to reduce the structure of the graph into component pieces including simply connected regions, regions enclosing a single leaf, and regions enclosing an additional interior cycle $C_{i+1}$ when in the presence of several interior cycles $C_i,C_{i+1},\ldots$ within the graph.  In general one must perform a sequence of such moves until all that is left is a single annulus around $C_i$, which will be considered in Lemma \ref{lem:indhyp} below.
%We begin by showing that each reduction move preserves the equality in Theorem \ref{conj:sum}.
%JERRY
%}

This completes the proof of the Main Theorem \ref{conj:sum}.
\end{proof}

\section{Examples}
\label{sec:examples}

\begin{example}
\label{ex:Abe}
Abe (in \cite{Abe}) considers a six crossing knot universe as a running example, displaying its clock lattice (with one edge missing) in Figure 7.  We translate this example to the graph of perfect matchings in Figure \ref{fig:ex:Abe} below.
\end{example}

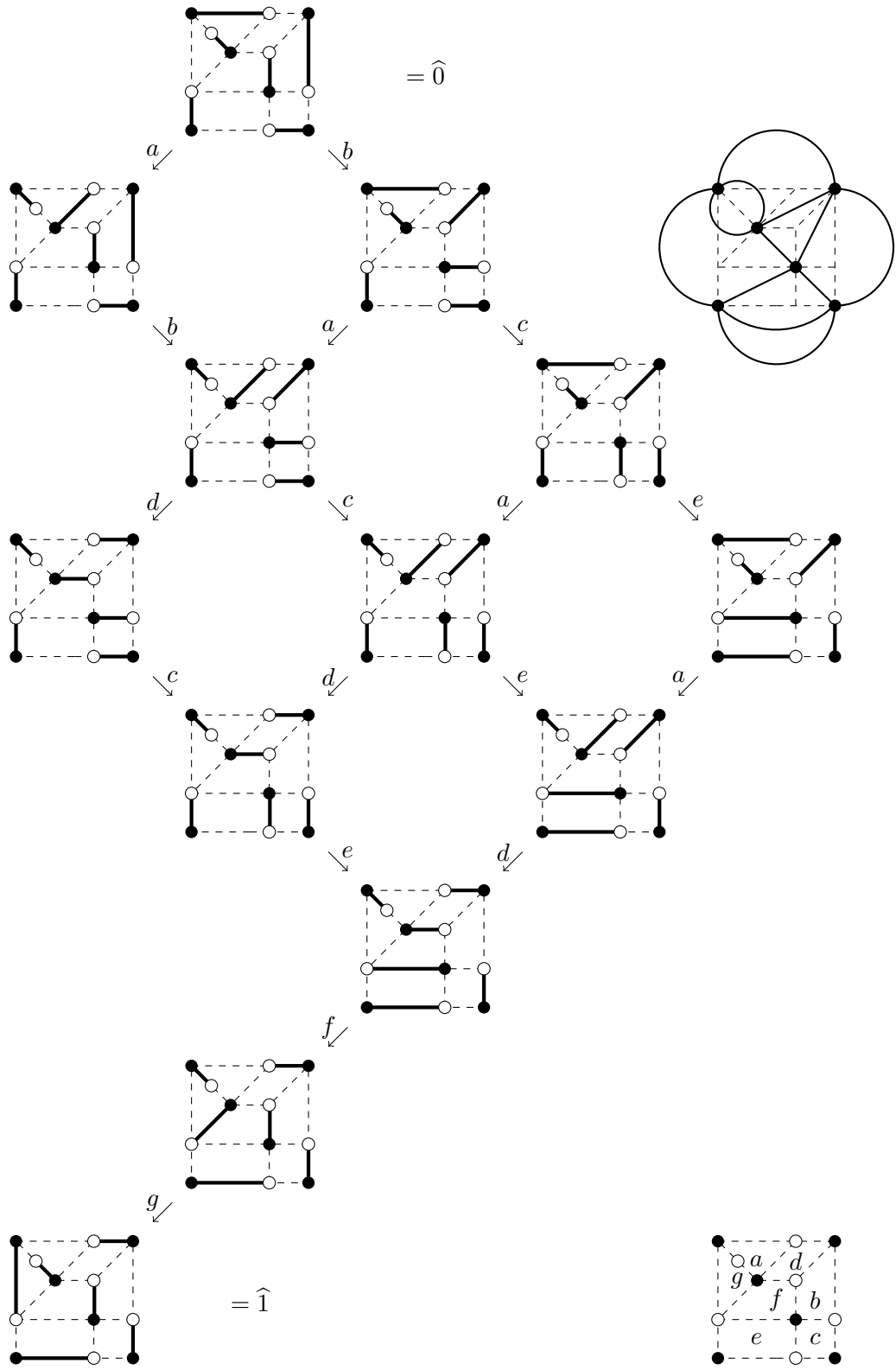
\begin{figure}
\begin{center}
\begin{tikzpicture}

\begin{scope}[scale=.9]

%\foreach \x/ \y in {0/0}
%	{
%	\draw[dashed] (\x,\y) -- (\x+1,\y) -- (\x+1,\y+2) -- (\x,\y+1.5) -- (\x-.5,\y+1.5) -- (\x-1,\y+2) -- (\x-1,\y) -- cycle;
%	\draw[dashed] (\x,\y) -- (\x,\y+1.5);
%	\draw[dashed] (\x+1,\y+1) -- (\x-1,\y+1) -- (\x,\y+2);
%	}

%Bottom row on the left
\draw[ultra thick] (-7,1) -- (-5.66,1);
\draw[shift={(3,3)}, ultra thick] (-7,1) -- (-5.66,1);
\draw[shift={(6,6)}, ultra thick] (-7,1) -- (-5.66,1);
\draw[shift={(9,9)}, ultra thick] (-7,1) -- (-5.66,1);
\draw[shift={(12,12)}, ultra thick] (-7,1) -- (-5.66,1);

%Bottom row on the right
\draw[shift={(0,12)}, ultra thick] (-5,1) -- (-5.66,1);
\draw[shift={(3,15)}, ultra thick] (-5,1) -- (-5.66,1);
\draw[shift={(6,18)}, ultra thick] (-5,1) -- (-5.66,1);
\draw[shift={(0,18)}, ultra thick] (-5,1) -- (-5.66,1);
\draw[shift={(3,21)}, ultra thick] (-5,1) -- (-5.66,1);

%Second row on the right (vertical)
\draw[ultra thick] (-5,1) -- (-5,1.66);
\draw[shift={(3,3)}, ultra thick] (-5,1) -- (-5,1.66);
\draw[shift={(6,6)}, ultra thick] (-5,1) -- (-5,1.66);
\draw[shift={(9,9)}, ultra thick] (-5,1) -- (-5,1.66);
\draw[shift={(12,12)}, ultra thick] (-5,1) -- (-5,1.66);
\draw[shift={(3,9)}, ultra thick] (-5,1) -- (-5,1.66);
\draw[shift={(6,12)}, ultra thick] (-5,1) -- (-5,1.66);
\draw[shift={(9,15)}, ultra thick] (-5,1) -- (-5,1.66);

%Second row in the middle
\draw[shift={(3,9)}, ultra thick] (-5.66,1) -- (-5.66,1.66);
\draw[shift={(6,12)}, ultra thick] (-5.66,1) -- (-5.66,1.66);
\draw[shift={(9,15)}, ultra thick] (-5.66,1) -- (-5.66,1.66);

%Second row on the left
\draw[shift={(3,9)}, ultra thick] (-7,1) -- (-7,1.66);
\draw[shift={(6,12)}, ultra thick] (-7,1) -- (-7,1.66);
\draw[shift={(9,15)}, ultra thick] (-7,1) -- (-7,1.66);
\draw[shift={(0,12)}, ultra thick] (-7,1) -- (-7,1.66);
\draw[shift={(3,15)}, ultra thick] (-7,1) -- (-7,1.66);
\draw[shift={(6,18)}, ultra thick] (-7,1) -- (-7,1.66);
\draw[shift={(0,18)}, ultra thick] (-7,1) -- (-7,1.66);
\draw[shift={(3,21)}, ultra thick] (-7,1) -- (-7,1.66);

%Third row on the left
\draw[shift={(6,6)}, ultra thick] (-7,1.66) -- (-5.66,1.66);
\draw[shift={(9,9)}, ultra thick] (-7,1.66) -- (-5.66,1.66);
\draw[shift={(12,12)}, ultra thick] (-7,1.66) -- (-5.66,1.66);

%Third row on the right
\draw[shift={(0,12)}, ultra thick] (-5,1.66) -- (-5.66,1.66);
\draw[shift={(3,15)}, ultra thick] (-5,1.66) -- (-5.66,1.66);
\draw[shift={(6,18)}, ultra thick] (-5,1.66) -- (-5.66,1.66);

%Fourth row on the left
\draw[ultra thick] (-7,1.66) -- (-7,3);

%Fourth row diagonal
\draw[shift={(3,3)}, ultra thick] (-7,1.66) -- (-6.33,2.33);

%Fourth row in the middle
\draw[ultra thick] (-5.66,1.66) -- (-5.66,2.33);
\draw[shift={(3,3)}, ultra thick] (-5.66,1.66) -- (-5.66,2.33);
\draw[shift={(0,18)}, ultra thick] (-5.66,1.66) -- (-5.66,2.33);
\draw[shift={(3,21)}, ultra thick] (-5.66,1.66) -- (-5.66,2.33);

%Fourth row on the right
\draw[shift={(0,18)}, ultra thick] (-5,1.66) -- (-5,3);
\draw[shift={(3,21)}, ultra thick] (-5,1.66) -- (-5,3);

%Middle
\draw[shift={(6,6)}, ultra thick] (-6.33,2.33) -- (-5.66,2.33);
\draw[shift={(3,9)}, ultra thick] (-6.33,2.33) -- (-5.66,2.33);
\draw[shift={(0,12)}, ultra thick] (-6.33,2.33) -- (-5.66,2.33);

%Fifth row on the left one
\draw[ultra thick] (-6.33,2.33) -- (-6.66,2.66);
\draw[shift={(12,12)}, ultra thick] (-6.33,2.33) -- (-6.66,2.66);
\draw[shift={(9,15)}, ultra thick] (-6.33,2.33) -- (-6.66,2.66);
\draw[shift={(6,18)}, ultra thick] (-6.33,2.33) -- (-6.66,2.66);
\draw[shift={(3,21)}, ultra thick] (-6.33,2.33) -- (-6.66,2.66);

%Fifth row on the left two
\draw[shift={(3,3)}, ultra thick] (-7,3) -- (-6.66,2.66);
\draw[shift={(6,6)}, ultra thick] (-7,3) -- (-6.66,2.66);
\draw[shift={(9,9)}, ultra thick] (-7,3) -- (-6.66,2.66);
\draw[shift={(3,9)}, ultra thick] (-7,3) -- (-6.66,2.66);
\draw[shift={(6,12)}, ultra thick] (-7,3) -- (-6.66,2.66);
\draw[shift={(0,12)}, ultra thick] (-7,3) -- (-6.66,2.66);
\draw[shift={(3,15)}, ultra thick] (-7,3) -- (-6.66,2.66);
\draw[shift={(0,18)}, ultra thick] (-7,3) -- (-6.66,2.66);

%Fifth row in the middle
\draw[shift={(9,9)}, ultra thick] (-5.66,3) -- (-6.33,2.33);
\draw[shift={(6,12)}, ultra thick] (-5.66,3) -- (-6.33,2.33);
\draw[shift={(3,15)}, ultra thick] (-5.66,3) -- (-6.33,2.33);
\draw[shift={(0,18)}, ultra thick] (-5.66,3) -- (-6.33,2.33);

%Fifth row in the middle
\draw[shift={(9,9)}, ultra thick] (-5.66,2.33) -- (-5,3);
\draw[shift={(12,12)}, ultra thick] (-5.66,2.33) -- (-5,3);
\draw[shift={(6,12)}, ultra thick] (-5.66,2.33) -- (-5,3);
\draw[shift={(9,15)}, ultra thick] (-5.66,2.33) -- (-5,3);
\draw[shift={(3,15)}, ultra thick] (-5.66,2.33) -- (-5,3);
\draw[shift={(6,18)}, ultra thick] (-5.66,2.33) -- (-5,3);

%Top row on the left
\draw[shift={(12,12)}, ultra thick] (-7,3) -- (-5.66,3);
\draw[shift={(9,15)}, ultra thick] (-7,3) -- (-5.66,3);
\draw[shift={(6,18)}, ultra thick] (-7,3) -- (-5.66,3);
\draw[shift={(3,21)}, ultra thick] (-7,3) -- (-5.66,3);

%Top row on the right
\draw[ultra thick] (-5,3) -- (-5.66,3);
\draw[shift={(3,3)}, ultra thick] (-5,3) -- (-5.66,3);
\draw[shift={(6,6)}, ultra thick] (-5,3) -- (-5.66,3);
\draw[shift={(3,9)}, ultra thick] (-5,3) -- (-5.66,3);
\draw[shift={(0,12)}, ultra thick] (-5,3) -- (-5.66,3);

\draw[<-] (-4.66,3.33) -- (-4.33,3.66);
	\draw (-4.66,3.66) node {$g$};
\draw[<-] (-1.66,6.33) -- (-1.33,6.66);
	\draw (-1.66,6.66) node {$f$};
\draw[<-] (-1.33,9.33) -- (-1.66,9.66);
	\draw (-1.33,9.66) node {$e$};
\draw[<-] (1.33,9.33) -- (1.66,9.66);
	\draw (1.33,9.66) node {$d$};

\draw[<-] (-4.33,12.33) -- (-4.66,12.66);
	\draw (-4.33,12.66) node {$c$};
\draw[<-] (-1.66,12.33) -- (-1.33,12.66);
	\draw (-1.66,12.66) node {$d$};
\draw[<-] (1.66,12.33) -- (1.33,12.66);
	\draw (1.66,12.66) node {$e$};
\draw[<-] (4.33,12.33) -- (4.66,12.66);
	\draw (4.33,12.66) node {$a$};

\draw[<-] (-4.66,15.33) -- (-4.33,15.66);
	\draw (-4.66,15.66) node {$d$};
\draw[<-] (-1.33,15.33) -- (-1.66,15.66);
	\draw (-1.33,15.66) node {$c$};
\draw[<-] (1.33,15.33) -- (1.66,15.66);
	\draw (1.33,15.66) node {$a$};
\draw[<-] (4.66,15.33) -- (4.33,15.66);
	\draw (4.66,15.66) node {$e$};

\draw[<-] (-4.33,18.33) -- (-4.66,18.66);
	\draw (-4.33,18.66) node {$b$};
\draw[<-] (-1.66,18.33) -- (-1.33,18.66);
	\draw (-1.66,18.66) node {$a$};
\draw[<-] (1.66,18.33) -- (1.33,18.66);
	\draw (1.66,18.66) node {$c$};

\draw[<-] (-1.33,21.33) -- (-1.66,21.66);
	\draw (-1.33,21.66) node {$b$};
\draw[<-] (-4.66,21.33) -- (-4.33,21.66);
	\draw (-4.66,21.66) node {$a$};

\draw (-3,2) node {$=\widehat{1}$};
\draw (0,23) node {$=\widehat{0}$};

%This is the example box:
\draw (5.66,1.33) node {$e$};
\draw (6.66,1.33) node {$c$};
\draw (6,2) node {$f$};
\draw (6.66,2) node {$b$};
\draw (5.33,2.33) node {$g$};
\draw (5.66,2.66) node {$a$};
\draw (6.33,2.66) node {$d$};

\foreach \x/ \y in {-6/1, -3/4, 0/7, -3/10, 3/10, -6/13, 0/13, 6/13, -3/16, 3/16, 0/19, -6/19, -3/22, 6/1}
	{
	\draw[dashed] (\x,\y) -- (\x+1,\y) -- (\x+1,\y+2) -- (\x+.33,\y+1.33) -- (\x-.33,\y+1.33) -- (\x-1,\y+2) -- (\x-1,\y) -- cycle;
	\draw[dashed] (\x+.33,\y) -- (\x+.33,\y+1.33);
	\draw[dashed] (\x+1,\y+.66) -- (\x-1,\y+.66) -- (\x+.33,\y+2);
	\draw[dashed] (\x-1,\y+2) -- (\x+1,\y+2);

	\foreach \w/ \z in {\x+.33/\y, \x-1/\y+.66, \x+1/\y+.66, \x+.33/\y+1.33, \x-.66/\y+1.66, \x+.33/\y+2}
		{	\fill[color=white] (\w,\z) circle (3pt);
			\draw (\w,\z) circle (3pt);
		}
	\foreach \w/ \z in {\x-1/\y, \x+1/\y, \x+.33/\y+.66, \x-.33/\y+1.33, \x-1/\y+2, \x+1/\y+2}
		{	\fill[color=black] (\w,\z) circle (3pt);
		}
	}

\foreach \x/ \y in {6/19}
	{
	\draw[dashed] (\x,\y) -- (\x+1,\y) -- (\x+1,\y+2) -- (\x+.33,\y+1.33) -- (\x-.33,\y+1.33) -- (\x-1,\y+2) -- (\x-1,\y) -- cycle;
	\draw[dashed] (\x+.33,\y) -- (\x+.33,\y+1.33);
	\draw[dashed] (\x+1,\y+.66) -- (\x-1,\y+.66) -- (\x+.33,\y+2);
	\draw[dashed] (\x-1,\y+2) -- (\x+1,\y+2);

%	\foreach \w/ \z in {\x+.33/\y, \x-1/\y+.66, \x+1/\y+.66, \x+.33/\y+1.33, \x-.66/\y+1.66, \x+.33/\y+2}
%		{	\fill[color=white] (\w,\z) circle (3pt);
%			\draw (\w,\z) circle (3pt);
%		}
	\foreach \w/ \z in {\x-1/\y, \x+1/\y, \x+.33/\y+.66, \x-.33/\y+1.33, \x-1/\y+2, \x+1/\y+2}
		{	\fill[color=black] (\w,\z) circle (3pt);
		}
	}

\begin{scope}[shift={(0,14.5)}]

\draw[thick] (5,4.5)--(6.33,5.16)--(7,4.5);
\draw[thick] (7,6.5)--(6.33,5.16)--(5.66,5.83)--(7,6.5);

%Use arcs to draw a and g.  then also for semicircles on outside.

\draw[thick] (5,4.5) arc (225:315:1.4);
\draw[thick] (5,4.5) arc (180:360:1);
\draw[thick] (5,6.5) arc (180:0:1);
\draw[thick] (5,4.5) arc (270:90:1);
\draw[thick] (7,4.5) arc (270:450:1);
\draw[thick] (5,6.5) arc (135:495:.462);

\end{scope}

\end{scope}

\end{tikzpicture}
	\caption{An example from Abe.}
	\label{fig:ex:Abe}
\end{center}
\end{figure}

\begin{example}
\label{ex:split}
\textbf{Split links and other non-prime-like diagrams.}  Given such a diagram, one can use a trick similar to one in \cite{VIR}, pulling some strand from one side over or under a strand on the other side.  Repeat several times if necessary to produce a prime-like diagram like those considered above.
\end{example}

%NOT SURE IF THIS IS RIGHT.
%\begin{example}
%\label{ex:nugatory}
%\textbf{Diagrams with nugatory crossings.}  As in the above example, one can pull some strand from one side of the nugatory crossing over or under a strand on the other side.  Repeat several times if necessary to produce a prime-like diagram.
%\end{example}

\begin{example}
\label{ex:10-44}
\textbf{The number of cycles is not a knot invariant.}  Consider for example the two projections of the knot $10_{44}$ found in Knot Info \cite{knotinfo} and the Knot Atlas \cite{knotatlas}.  Taking ``similar'' starred regions, the balanced overlaid Tait graph of the first has two cycles but that of the second has only one.  The graph of perfect matchings of the first has height $11+1=12$ and that of the second has height $9$.
\end{example}

However, the number of cycles may be used to obtain an upper bound for the bridge number of the knot if the following conjecture is true.

\begin{conjecture}
\label{conj:bridge}
The number of cycles is related to the bridge number of the universe.
\end{conjecture}

This is supported by notions in the next subsection.

%{\red 
%MAYBE say something about Connect Sum???
%}

%\section{Results}
%\sub
\subsection{Application to grid graphs and harmonic knots}
\label{subsec:harmonic}

Let $\Gamma_{m,n}$ be the grid graph of $m\times n$ squares with $m$, $n$ both odd.  This ensures that the graph has equal-sized vertex sets.  Let us say that $m\leq n$, so that $\min=\min\{m,n\}=m$.  Then $(1/2)(\min-1)=\frac{m-1}{2}$, $(1/2)(\min-1)+1=\frac{m+1}{2}$, and $2(1/2)(\min-1)+1=m$.

\begin{corollary}
The height of the clock lattice for the square grid graph $\Gamma_{m,m}$ is the $m$-th tetrahedral number $(1/6)m(m+1)(m+2)$.
\end{corollary}

\begin{proof}
Let $\min=\min\{m,n\}$.  By Theorem \ref{conj:sum} above, the height of the clock lattice is:
\begin{align*}
  &= \sum_{i=0}^{(1/2)(\min-1)} (m-2i)(n-2i) = \sum_{i=0}^{(1/2)(\min-1)} [mn -2i(m+n)+4i^2] \\
	&= \frac{mn(\min+1)}{2}-\frac{(m+n)(\min-1)(\min+1)}{4}+\frac{\min(\min-1)(\min+1)}{6} \\
  &= \frac{(3m^2n+6mn-m^3+m+3n)}{12}\text{ setting $\min=m$} \\
  &= \frac{(m+1)[3n(m+1)-m(m-1)]}{12} \\
  &= \frac{(m)(m+1)(m+2)}{6}\text{ in the case where $n=m$.}
\end{align*}

Observe that this is the $m$-th tetrahedral number, which is [A000292] on the On-Line Encyclopedia of Integer Sequences \cite{OEIS}.
\end{proof}

A \emph{harmonic curve} is one that admits a parametrization whose three coordinate functions $x=x_1,y=x_2,z=x_3$ are the classical Chebyshev polynomials $T_{x_i}(t)$ defined by $T_n(\cos t)=\cos (nt)$.  A \emph{Chebyshev curve} is one whose third coordinate function $T_{x_3}(t+\varphi)$ has a phase shift.  Identifying the ends of a non-singular harmonic curve, one obtains a \emph{harmonic knot} $H(x_1,x_2,x_3)$ if and only if the three parameters $x_i$ are pairwise coprime integers (Comstock 1897 \cite{Comstock} or see also \cite{KosPec3} or \cite{FF}).

Koseleff and Pecker found in \cite{KosPec1} that the trefoil could be parametrized in such a way, leading them to study harmonic knots in \cite{KosPec4}, \cite{KosPec3}, and \cite{KosPec}.  Harmonic knots are polynomial analogues of the famous Lissajous knots studied in \cite{BDHZ}, \cite{BHJS}, \cite{Crom}, \cite{HosZir}, \cite{JonPrz}, \cite{Lam}, \cite{Lam:dis}; however, the figure-eight knot is not a Lissajous knot but is the harmonic knot $H(3,5,7)$.

The interested reader may find many examples of Chebyshev knot diagrams on the website of Vincent-Pierre Koseleff:  \verb=http://www.math.jussieu.fr/~koseleff/knots/kindex1.html=

\begin{theorem}
\label{thm:cheb}
%(Koseleff-Pecker 
\cite[Theorem 3]{KosPec3} %)
Every knot has a projection which is a Chebyshev plane curve.
\end{theorem}

This is a consequence of another result in their paper involving bridge number, further confirming the suspicions of Conjecture \ref{conj:bridge}.

\begin{remark}
The harmonic knot $H(x_1,x_2,x_3)$ has a balanced overlaid Tait graph $\Gamma$ that is the $(x_1-2)\times(x_2-2)$ grid graph $\Gamma_{x_1-2,x_2-2}$.  Thus every knot has a projection whose balanced overlaid Tait graph $\Gamma$ is a grid graph!
\end{remark}

%{\red pull some figures from their website...}

We conclude with one other application.

\subsection{Relationship with discrete Morse theory}
\label{subsec:DMT}

This subject was introduced in 1995 by Forman (see for example \cite{For}) to apply the power of the classical version to combinatorially defined complexes.  The critical points found in the smooth version can be determined combinatorially by \emph{collapsing} pairs of $i$- and $i+1$-dimensional cells in the complex $C$, and these collapses are described (and ordered) using the following map.

A \emph{discrete Morse function} is a weakly increasing map $f:(C,\subseteq)\rightarrow(\mathbb{Z},\leq)$ such that $|f^{-1}(n)|\leq 2$ for all $n\in\mathbb{Z}$ and such that $f(\sigma)=f(\tau)$ implies that one of $\sigma$, $\tau$ is a face of the other.  A \emph{critical cell} for a discrete Morse function $f$ is face of $C$ at which $f$ is injective.

This subject has been studied in many contexts, including by Chari in \cite{Cha} for shellability, by Babson and Hersh in \cite{BabHer} for lexicographic orders, by Kozlov in \cite{Koz} for free chain complexes, by Welker in \cite{Wel} for free resolutions, by Engstr\"om in \cite{Eng} for Fourier transforms, by Ayala, Fernández, Fernández-Ternero, and Vilches in \cite{AFFTV} for graphs, by Benedetti in \cite{Ben} for homology, and by Salvetti, Gaiffi, and Mori in \cite{SalGai}, \cite{SalGaiMor} for line arrangements and in \cite{SalMor} for configuration spaces.  It was recently the topic of a summer school at the Institut Mittag-Leffler outside of Stockholm, Sweden organized by Benedetti and Engstr\"om.

Consider a 2-complex $\Delta$ of the 2-sphere, not necessarily simplicial, as we allow multiple 1-faces with the same two endpoints, 1-faces with the same single endpoint, and 2-faces of length less than three.  This can be realized as a(n unsigned) plane graph $G$.  Associate to this a knot universe $U$ described above.  Then the face poset $\mathcal{F}(\Delta)$ of the complex can be realized as $\widehat{\Gamma}$.  Following discrete Morse theory, a(n elementary) collapse on $\Delta$ can be realized as an edge in $\widehat{\Gamma}$.  The deletion of the two starred regions in $\widehat{\Gamma}$ to obtain $\Gamma$ corresponds with choosing the two critical cells, one of dimension two and the other of dimension one.

Then a shelling of $\Delta$ with chosen critical cells corresponds with a discrete Morse function $f$ on $\Delta$ with these chosen critical cells, which corresponds with a perfect matching of $\Gamma$, which corresponds with two rooted spanning trees $T$ of $G$ and $T^C$ of the dual graph $G^*$, which correspond with a series of collapses together with a series of \emph{endocollapses} (collapsing in the dual).

%{\red DISCRETE MORSE THEORY:  A 2-complex is $\widehat{\Gamma}$ is the face poset of the complex [not nec. simplicial since might have 2-faces of length 2] of the 2-sphere given by the universe.  Edges in perfect matchings are elementary collapses, perfect matchings themselves are shellings.  What does it mean for two shellings to create a cycle?}

%{\red As it turns out, a perfect matching of $\Gamma$ is actually a discrete Morse function with predicted critical cells on a 2-complex of $S^2$ that has $G$ as its 1-skeleton.}

%We believe that the techniques below could be used to understand the space of all Morse functions of a given triangulation of the 2-sphere.  We expect to use these ideas to generalize to general surfaces.  Furthermore we believe these techniques might be used to study $d$-dimensional knotting in $2d$-dimensional manifolds.

\begin{proposition}
\label{prop:DMT}
%Discrete Morse functions on the 2-complex $\Delta$ of the 2-sphere correspond to perfect matchings on $\Gamma$ constructed from a knot diagram $D$.
Perfect matchings on $\Gamma$ constructed from a knot diagram $D$ correspond to discrete Morse functions on the 2-complex $\Delta$ of the 2-sphere.
\end{proposition}

\begin{proof}
%This follows by construction with the 1-skeleton of $\Delta$ giving the Tait graph $G$ of the knot diagram and with the face poset $\mathcal{F}(\Delta)$ of the complex giving $\widehat{\Gamma}$.
This follows by construction with the Tait graph $G$ coming from the 1-skeleton of $\Delta$ and with $\widehat{\Gamma}$ as the face poset $\mathcal{F}(\Delta)$ of the complex.
\end{proof}

%And... 36 references!  +1 for Greene +1 for VIR +2 atlas and info = 40
% NOW total of 1 + 32 + 21 = 54 references!  wow!  +2 = 56

\bibliographystyle{amsalpha}
\bibliography{12OctBibliography}

\end{document}